\documentclass[12pt]{extarticle}
\usepackage{amsmath, amsthm, amssymb, mathtools, hyperref,color}
\usepackage{graphicx}
\usepackage{comment}
\usepackage{bbm}

\oddsidemargin \evensidemargin
\newtheorem{theorem}{Theorem}
\numberwithin{theorem}{section}
\newtheorem{proposition}[theorem]{Proposition}
\newtheorem{lemma}[theorem]{Lemma}
\newtheorem{corollary}[theorem]{Corollary}
\newtheorem{definition}[theorem]{Definition}
\newtheorem{remark}[theorem]{Remark}
\newtheorem{conjecture}[theorem]{Conjecture}
\newtheorem{problem}[theorem]{Problem}
\newtheorem{example}[theorem]{Example}

\newcommand{\R}{\mathbb{R}}
\newcommand{\N}{\mathbb{N}}
\newcommand{\C}{\mathbb{C}}

\newcommand{\PP}{\mathbb{P}}

\newcommand{\cX}{\mathcal{X}}

\newcommand{\dd}{{\rm d}}

\date{}
\def\S{\mathbb{S}}
\title{\textbf{Exponential Varieties}
\author{ Mateusz Micha{\l}ek,
Bernd Sturmfels, \\
Caroline Uhler, and Piotr Zwiernik }}

\begin{document}

\maketitle

\begin{abstract} \noindent
Exponential varieties arise from exponential families in statistics.
These real algebraic varieties have strong positivity
and convexity properties, familiar from
toric varieties and their moment maps.
Among them are 
varieties of inverses of symmetric matrices satisfying linear constraints.
This class includes Gaussian graphical models.
 We develop a general theory of exponential varieties. These are derived
from hyperbolic polynomials and their
integral representations.
We compare the multidegrees and ML degrees
of the gradient map for hyperbolic polynomials.
\end{abstract}

\section{Introduction}
\label{intro}

An exponential variety is a real algebraic variety in
an affine or projective space that has a distinguished
set of positive points. This positive part maps
bijectively to a convex body. The inverse
to this map is an algebraic function from the convex body to
the variety.

This situation is familiar to geometers
from the theory of toric varieties. The convex body is
a polytope or polyhedral cone which is in bijection
with the positive part of the toric variety via the
moment map. The inverse of the moment map
is an algebraic function whose degree
is the Euler characteristic of an associated
hypersurface in the torus \cite[Thm.~3.2]{HS}.

A similar situation happens for the image $\mathcal{L}^{-1}$ of a linear
subspace $\mathcal{L} \subset \R^n$ under the
coordinatewise inversion map
$(x_1,\ldots,x_n) \mapsto (x_1^{-1}, \ldots , x_n^{-1})$.
Suppose  the convex polyhedron $\mathcal{L}_{\succ 0} = \R_{> 0}^n \cap \mathcal{L}$
is non-empty. The {\em reciprocal linear space} $\mathcal{L}^{-1}$
is a subvariety of $\R^n$ whose positive part
$\mathcal{L}^{-1}_{\succ  0}$ admits a canonical bijection
to the polyhedron dual to  $\mathcal{L}_{\succ 0}$. By composing with
the inversion map $\mathcal{L} \dashrightarrow \mathcal{L}^{-1}$,
and by passing from $\R^n$ to the projective space $\R\PP^{n-1}$, one obtains
a canonical bijection between the interior of any polytope and its dual;
see Example \ref{ex_prodlinform}.
In \cite[Example 3.5]{stuhl}, this process was described as
``how to morph a cube into an octahedron''.

In this paper we present a general theory that explains
these examples and many more.
Our point of departure is the theory of {\em exponential families}
in statistics \cite{brown}. Exponential varieties arise from exponential
families. Known results about the latter explain the remarkable
positivity properties of the former. In Section~\ref{sec:rat_exp_fam}
we review  basics on exponential families, such as the bijection between
the canonical parameters and the sufficient statistics (Theorem~\ref{thm:brown}).
Two key examples  are the multivariate
Gaussian family and the full discrete exponential family.
New points in Section~\ref{sec:rat_exp_fam} are the connection to
barrier functions in convex optimization (cf.~\cite{guler1}) and our definition of
{\em rational exponential families}. These statistical models have the
property that
the gradient of their {\em log-partition} function is given by rational functions,
a property much stronger than that required by
Drton and Sullivant \cite{DS} for   {\em algebraic exponential families}.

In Section \ref{sec:dis_gau_hyp} we define {\em hyperbolic exponential families}. These are associated with hyperbolic polynomials and their hyperbolicity cones~\cite{guler2,KPV}. Work of Scott and Sokal~\cite{sokal} implies that every exponential family whose canonical parameters form a convex cone and whose partition function is the power of a homogeneous polynomial must be hyperbolic (Theorem~\ref{th:gardingconverse}). We conjecture that the converse holds as well, namely that one can build a (statistical) exponential family, i.e.~with an underlying measure, from any hyperbolic polynomial (Conjecture~\ref{conj_Riesz}). 

The formal definition of {\em exponential varieties} appears in Section \ref{sec:lin_sub}.
They arise by fixing a linear subspace $\mathcal{L}$ in the space of canonical parameters.
The variety is the Zariski closure of the image of $\mathcal{L}$ under the
gradient map of the log-partition function.  In the discrete case, $\mathcal{L}$
has to be defined over $\mathbb{Q}$ in order to get an algebraic object.
We focus on the hyperbolic case, where every linear subspace that meets the
parameter cone defines an exponential variety. This includes the
Gaussian case, and we recover the {\em linear concentration models}~of~\cite{andersonLinearCovariance,stuhl}.

In Section \ref{sec:ml_degree} we examine the gradient
multidegree of a hyperbolic  exponential family. This is the
cohomology class of the graph of the gradient map.
Its coefficients characterize the degrees
of the exponential varieties defined by generic linear subspaces $\mathcal{L}$.
 Theorem \ref{eq:twoinequalities1} furnishes inequalities relating
  these degrees and the ML degree for special subspaces $\mathcal{L}$. Elementary symmetric form an important family of hyperbolic polynomials and Section~\ref{sec:ele_sym_pol} offers a detailed study of
the exponential varieties that they define. 

Section~\ref{sec:hankel} is devoted to Hankel models.
We show that the Grassmannian of lines, in its Pl\"ucker embedding, has the
structure of an exponential variety. This arises from the Gaussian case
by taking the linear space $\mathcal{L}$ of all Hankel matrices. The
extension  to (generalized) Hankel matrices of multivariate polynomials is related
to optimization via sums of squares~\cite{BPT}.

We close the introduction with a running example that will illustrate the general theory.

\begin{example} \rm
\label{ex:E_3}
Consider the  elementary symmetric polynomials in four parameters:
$$
\begin{matrix}
E_3(\theta) &  = & \theta_1 \theta_2 \theta_3
+ \theta_1 \theta_2 \theta_4
+ \theta_1 \theta_3 \theta_4
+ \theta_2 \theta_3 \theta_4, \\
E_2(\theta) &  = & \theta_1 \theta_2
+ \theta_1 \theta_3
+ \theta_1 \theta_4
+ \theta_2 \theta_3
+ \theta_2 \theta_4
+ \theta_3 \theta_4 , \\
E_1(\theta) &  = & \theta_1 + \theta_2 + \theta_3  + \theta_4.
\end{matrix}
$$
The  polynomial $E_3$ is hyperbolic with respect to
$\,C = \{ \theta \in \R^4 : E_i(\theta) > 0 \,\,\hbox{for}\,\, i=1,2,3\}$.
This set $C$ is the convex cone over the $3$-dimensional body known as the {\em elliptope},
and shown on the left in Figure~\ref{fig:somosa}. The closed dual cone
$K = C^\vee$ has
as its base the  convex hull of the {\em Steiner surface},
shown in the middle of Figure~\ref{fig:somosa},
which is the zero set of the quartic
\begin{equation}
\label{eq:steiner}
 Q \,\,=\,\,\sum \sigma_i^4 \,-\,4 \sum \sigma_i^3 \sigma_j
\,+\, 6 \sum \sigma_i^2 \sigma_j^2 \,+\,4 \sum \sigma_i^2 \sigma_j \sigma_k
\,-\, 40 \,\sigma_1 \sigma_2 \sigma_3 \sigma_4,
\end{equation}
where $\sigma_{1},\sigma_{2},\sigma_{3},\sigma_{4}$ are coordinates for the dual $\R^{4}$.

\begin{figure}[t]
\centering
 \includegraphics[width=4.9cm]{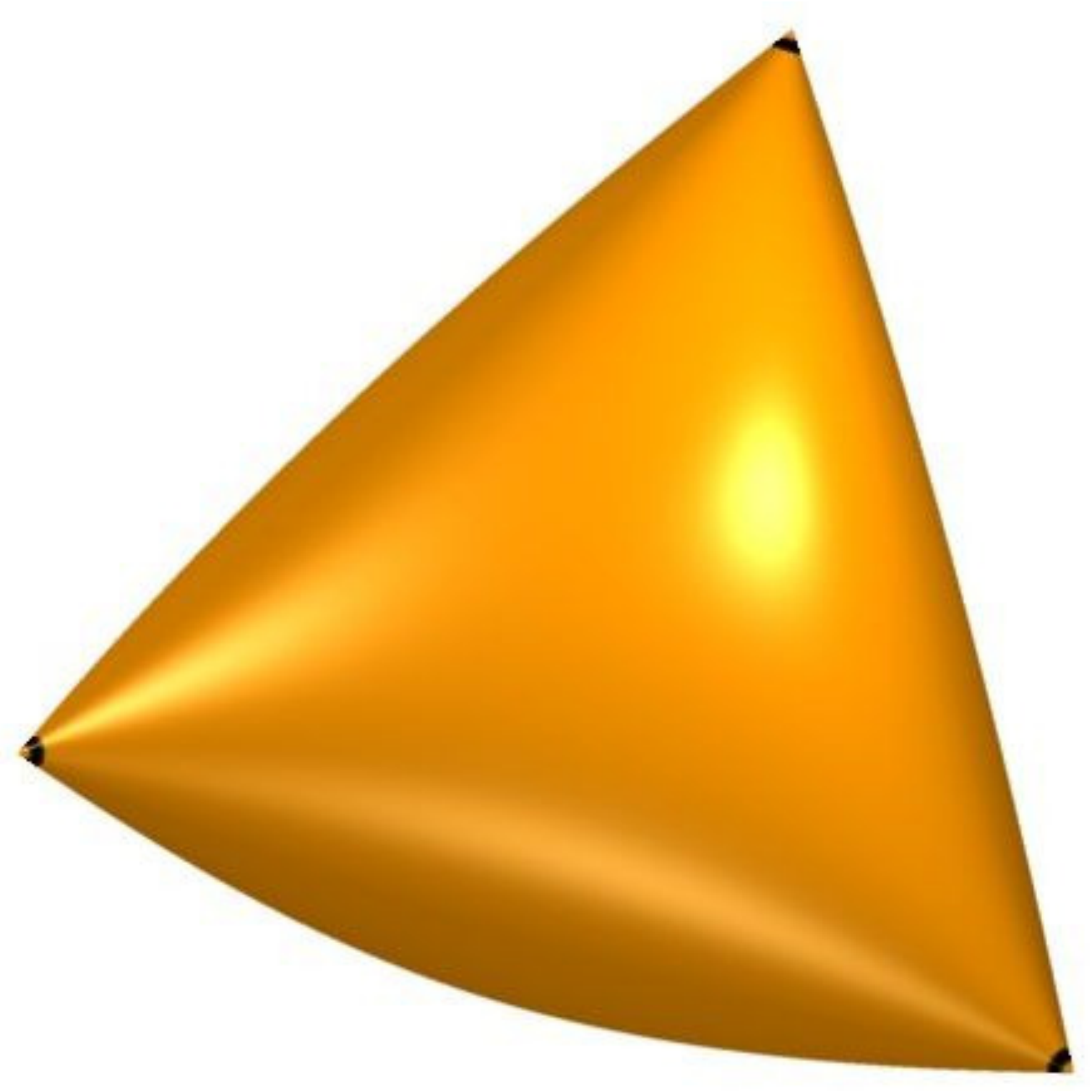}\includegraphics[width=4.9cm]{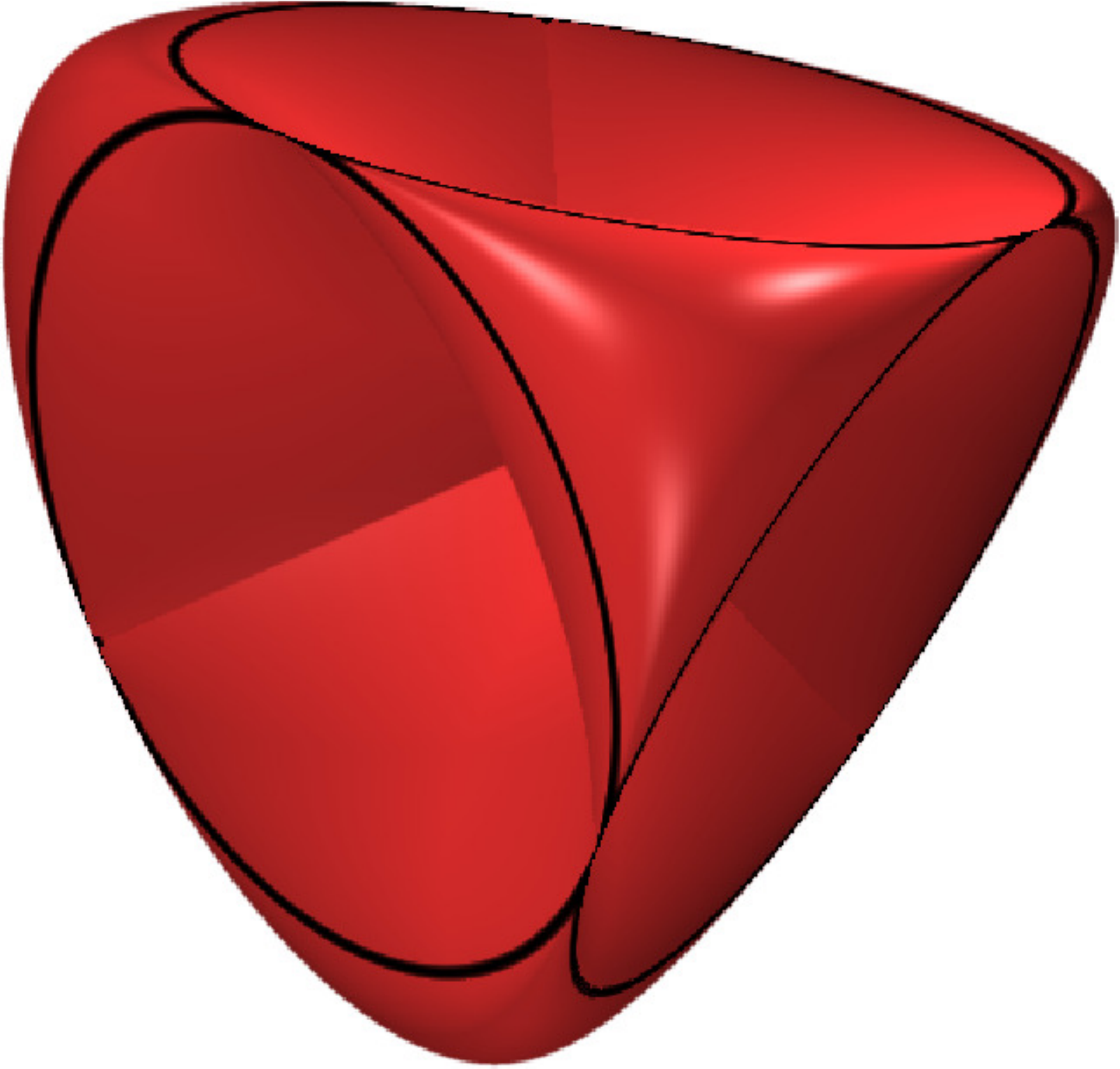}\includegraphics[width=4.8cm]{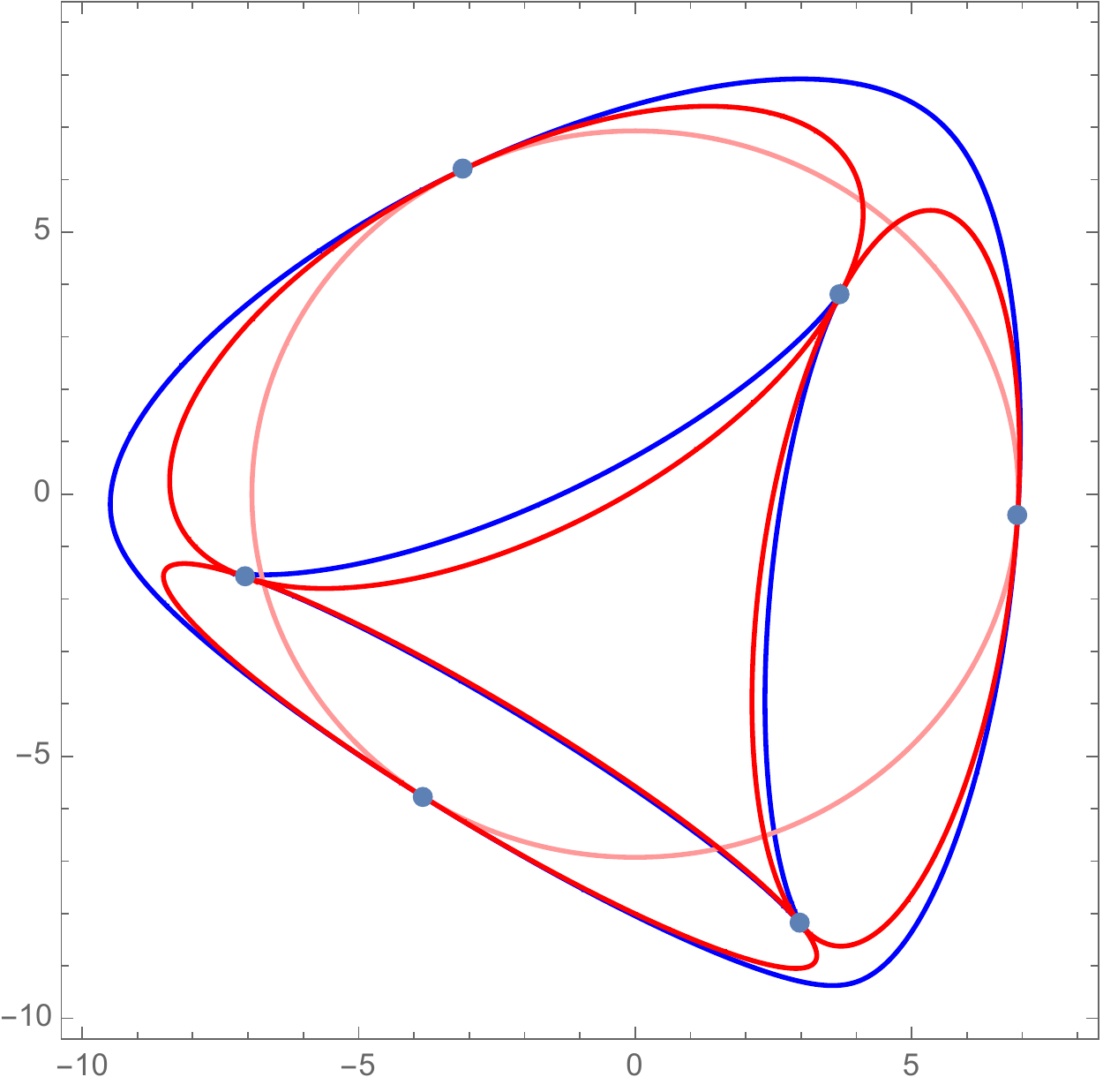}
\caption{The gradient of the elementary symmetric polynomial $E_3$ gives a bijection between the convex bodies of canonical parameters $C$ (left) and of sufficient statistics $K$~(middle).
The sextic branch curve of a projection of $K$
in the plane touches the four ellipses~(right).}
\label{fig:somosa}
\end{figure}

Here, our algebraic exponential family is defined by the gradient map of the cubic $E_3$:
$$ \nabla E_3 :  \R^4 \rightarrow \R^4 ,\,
\theta \mapsto  (
 \theta_2 \theta_3 {+} \theta_2 \theta_4 {+} \theta_3 \theta_4,
 \theta_1 \theta_3 {+} \theta_1 \theta_4 {+} \theta_3 \theta_4,
 \theta_1 \theta_2 {+} \theta_1 \theta_4 {+} \theta_2 \theta_4,
 \theta_1 \theta_2 {+} \theta_1 \theta_3 {+} \theta_2 \theta_3 ).  $$
 This map restricts to a bijection between $C$ and the interior of $K$.
The degree of $\nabla E_3$ is four, i.e.~generic
 points  $\sigma \in \C \PP^3$ have four preimages under the induced map
 $\nabla E_3: \C \PP^3 \dashrightarrow \C \PP^3$.
  But if $\sigma$ lies in the interior of $K$,
 then precisely one of these preimages lies in~$C = {\rm int}(K^\vee)$.

The map $\nabla E_3$ transforms the boundary of $C$ into the boundary of $K$ as follows.
The four circles in $\partial K$ are created by blowing up
the four vertices of $\partial C$, and the six points where they touch
are obtained by blowing down the six edges on $\partial C$.
The tetrahedron spanned by the vertices of $C$ is mapped to the
octahedron spanned by these six touching points in $\partial K$.

  We now fix a hyperplane $\mathcal{L}$
  in $\R^4$ that intersects the interior of $C$; see Figure~\ref{fig:yellowgreen}.
  The inclusion of $\mathcal{L} \simeq \R^3$ in $\R^4$ is dual to a projection
$\pi_\mathcal{L} : \R^4 \rightarrow  \R^3$.
The image of the cone $K$ under this map is dual to that seen in $C$ by intersecting with~$\mathcal{L}$,
i.e.~$\, \pi_\mathcal{L}(K) \,= \, (\mathcal{L} \cap C)^\vee $.
In the picture, we are cutting the body $C$ by
a plane, and dually we project $K$ into a plane.

    The image of $\mathcal{L}$ under
  $\nabla E_3$ is a quartic surface in $\C \PP^3$; see Figure~\ref{fig:redblue}.  That
  complex  surface is our {\em exponential
  variety}, denoted $\mathcal{L}^{\nabla E_3}$.
  The intersection curve $\,\{Q = 0\} \,\cap \,\mathcal{L}^{\nabla E_3} \,$
  decomposes into two components  having degrees $6$ and $10$.
  The sextic curve is smooth    when $\mathcal{L}$ is generic.
   It is the ramification locus of
  the Steiner surface $\{Q = 0\}$ under the map
 $\pi_\mathcal{L} : \C \PP^3 \dashrightarrow \C \PP^2$.

We now come to the punchline.
The  exponential variety $\mathcal{L}^{\nabla E_3}$ has
a distinguished positive part, namely the semialgebraic set
$\mathcal{L}^{\nabla E_3}_{\succ 0} = \mathcal{L}^{\nabla E_3}  \cap {\rm int}(K)$.
This comes with natural bijections
$$ \mathcal{L} \cap C  \,\,\longleftrightarrow\,\,
\mathcal{L}^{\nabla E_3}_{\succ 0}
 \,\, \longleftrightarrow \,\,{\rm int}(\pi_\mathcal{L}(K) ), $$
 to be derived  in Theorem  \ref{thm:sec4main}.
The intersection of the left body $C$ with a plane in Figure~\ref{fig:somosa} maps
bijectively to a quartic surface inside the middle body $K$.  That surface maps
bijectively onto the two-dimensional convex set
shown on the right in Figure~\ref{fig:somosa}. The branch curve is a plane curve of degree six, namely
the image of the ramification sextic. The four ellipses
in Figure~\ref{fig:somosa} (right) are the projections of the  circles in Figure~\ref{fig:somosa} (middle).
The six points are the projection of the vertices of the octahedron. We shall return to this
 in Example~\ref{ex:E_3c} and in Figure~\ref{fig:yellowgreen}.
\hfill $\diamondsuit$
\end{example}

\section{Exponential Families}
\label{sec:rat_exp_fam}

A statistical model is a family of probability distribution functions or probability density functions on
a sample space $(\cX,\nu,T)$. Here $\nu$ is the base measure on the set  $\cX$,
and $T: \cX \to \R^d$ is a measurable function
that represents the {\em sufficient statistics}. We fix an inner product
$\langle \,\cdot\,,\,\cdot\,\rangle$ on $\R^{d}$.
In the standard basis,  $\langle \theta , \sigma \rangle= \theta^{T} S \sigma$ for some positive definite symmetric matrix~$S$.
  An {\em exponential family} is a parametric statistical model on $(\cX,\nu,T)$ of the form
$$ p_\theta(x) \,\, = \,\, h(x) \cdot {\rm exp} \bigl( \langle \theta,T(x)\rangle - A(\theta) \bigr). $$
Here $h: \cX \rightarrow \R$ is a weight function which serves
to modify the given base measure $\nu$.
Without loss of generality, we can replace $\nu$ by $h(x) \cdot \nu$.
Afterwards we have $h(x)\equiv 1$. To simplify notation in later sections,
the following non-standard formulation will now be used.
 We replace $T(x)$ by $-T(x)$ and write all our exponential families in the form
\begin{equation}
\label{eq:expfam}
 p_\theta(x) \,\, = \,\,  {\rm exp} \bigl( -\langle \theta,T(x)\rangle - A(\theta) \bigr).
 \end{equation}
 The sample space $\cX$ itself is less important
 than its image $T(\cX)$ in $\R^d$. The reason is that  in (\ref{eq:expfam})
 we see the samples $x \in \cX$  only through their
  sufficient statistics $T(x) $. So, since
 $\nu$ induces a measure on $T(\cX)$, we might as well
 regard the set $T(\cX)$ as our space of data.

The function $A(\theta)$ is known as the {\em log-partition function}, or as the \emph{cumulant generating function}.
 It is uniquely determined by $(\cX,\nu,T)$ via the requirement that
$\int_\cX p_\theta(x) \nu(\dd x) = 1$.
This equation implies the following integral representation of the log-partition function:
\begin{equation}\label{eq_logpartition}
A(\theta)\;\;=\;\;\log \int_{\cX}\exp \bigl(-\langle \theta,T(x)\rangle \bigr)\,\nu(\dd x).
\end{equation}
The exponential of $A(\theta)$ is the {\em partition function} of the exponential family.

We are interested in the set of all parameters $\theta $ for which this integral is finite:
\begin{equation}
\label{eq:defC}
C \,\,:=\,\,\bigl\{\theta\in \R^{d}:\,\,A(\theta)<+\infty\bigr\}.
\end{equation}
By \cite[Theorem 1.13]{brown}, the set $C$ is convex in $\R^d$.
The function $A:\, C\to \R$ is convex, and it is
analytic on the interior of $C$. The convex set $C$ is the \emph{space of canonical parameters}. An exponential family is  said
to be \emph{regular} if $C$ is open in $\R^{d}$. It is  called \emph{minimal} if $T(\cX)$ spans the whole $\R^{d}$. See~\cite[Chapter 1]{brown} for standard definitions and basic properties of exponential families. We restrict our attention to exponential families that are minimal and regular.

Let $D A(\theta)$ be the differential of $A(\theta)$.
We view this as a linear form on $\R^{d}$. As such it lives in the dual
$\R^d$ with coordinates $\sigma = (\sigma_1,\ldots,\sigma_d)$ and plays a crucial role in this paper.

\begin{definition}[({\em Gradient map})] \rm
 Let $\theta^{*}\in \R^{d}$ be the unique vector such that $(D A(\theta))(\eta)=\langle \eta,-\theta^{*}\rangle$ for all $\eta\in \R^{d}$. The map $F:\R^{d}\to \R^{d}$, $\theta\mapsto \theta^{*}$ is called the \emph{gradient map}
 of $A(\theta)$.
\end{definition}

If $\langle\,\cdot\,,\,\cdot \,\rangle$ is the standard scalar product
then $\theta^{*}=-\nabla A(\theta)$, which explains the name. Formally, the gradient map maps $\R^{d}$ to its dual. The left $\R^d$ has coordinates $\theta$,
the right $\R^d$ has coordinates $\sigma$, and the pairing is the inner product
$\,\langle\theta, \sigma\rangle$.
The sufficient statistics $T(x)$ are elements of  $(\R^d,\sigma)$.
In that space we consider the convex hull of all sufficient statistics:
\begin{equation}
\label{eq:defK}
K\,\,:=\,\,{\rm conv}(T(\cX)) \,\,\, \subset \,\,\, \R^d.
\end{equation}
The following result is standard in the theory of exponential families (see~\cite[Theorem 3.6]{brown}).

\begin{theorem}
\label{thm:brown}
Let $(\cX,\nu,T)$ be an exponential family that is minimal and regular. Then the
gradient map $F$ defines an analytic bijection from $C$ to the interior of~$K$.
\end{theorem}

For extra clarity, let us paraphrase: the
space $C$ of canonical parameters is a convex set in the primal $(\R^d,\theta)$,
while the space $K$ of sufficient statistics is a convex set in the dual $(\R^d,\sigma)$.
The gradient of the log-partition function defines a bijection
between $C$ and ${\rm int}(K)$.
This paper concerns situations when this
bijection has desirable algebraic properties.

\begin{definition}[({\em Rational exponential family})]\label{df:REF} \rm
We say that an exponential family $(\cX,\nu,T)$ is
{\em rational} if the gradient map $F$
is a rational map $\,\R^d\dashrightarrow\R^d$.
\end{definition}

The case of primary interest to us arises when
the partition function is a negative power of a homogeneous polynomial $f$ in
the unknowns
$\theta= (\theta_1,\ldots,\theta_d)$. Suppose that $\langle\cdot,\cdot\rangle$ is the standard scalar product. In that case we have $A(\theta)=-\alpha \log f(\theta)$, where
$\alpha>0$ is a constant, and the gradient map
 is the rational~function
\begin{equation}
\label{eq:ratlmap1}
F \,:\,
\quad \R^d \dashrightarrow \R^d \,:\,\,\,
 \theta \,\mapsto \,\frac{\alpha}{f(\theta)}
\cdot
\bigl(\frac{ \partial f}{ \partial \theta_1},
\frac{ \partial f}{ \partial \theta_2}, \ldots,
\frac{ \partial f}{ \partial \theta_d} \bigr).
 \end{equation}
 For an algebraic geometer it is more natural to view this as a map
 of projective spaces:
 \begin{equation}
 \label{eq:ratlmap2}
F \,:\,
\C\PP^{d-1} \dashrightarrow \C\PP^{d-1} \,:\,\,
\theta \,\mapsto \,
\bigl(\frac{ \partial f}{ \partial \theta_1}:
\frac{ \partial f}{ \partial \theta_2}:\cdots:
\frac{ \partial f}{ \partial \theta_d} \bigr).
 \end{equation}

\begin{example}
\label{ex:E_3f} \rm
If $d = 4$ and $f = E_3$ is the third elementary symmetric polynomial,
then this is Example~\ref{ex:E_3}. Here $C$ and $K$ are
dual convex cones in $\R^4$, visualized in Figure~\ref{fig:somosa}
by their $3$-dimensional cross sections.  These live in the real part $\R \PP^3$
of $\C\PP^3$. Theorem \ref{thm:brown} tells us that the map
(\ref{eq:ratlmap2}) restricts to a bijection between
the interiors of the two convex bodies.
\hfill $\diamondsuit$
\end{example}

We next present two models that are fundamental in statistics and its applications.

\begin{example}[({\em Multivariate Gaussian family})]
\label{ex:fullgaussian}
\rm
Let $\cX=\R^{m}$, where $\nu$ is the Lebesgue measure, and set
$T(x)=\frac{1}{2}\,x  x^T \in \R^{m\times m}$.
 Here $d={m+1\choose 2}$ and we identify $\R^{d}$ with the space $\S^{m}$ of
 real symmetric $m\times m$ matrices with inner product $\langle A,B\rangle={\rm tr}(AB)$. So, the sufficient statistics are symmetric matrices of rank $\leq 1$. For this exponential family,
    the convex set $C$ in (\ref{eq:defC}) is the cone of
  positive definite matrices in $\S^{m}$, which we denote by ${\rm PD}_{m}$. This cone is open and selfdual.
  The set $K$ in (\ref{eq:defK}) is its closure --- the cone of positive semidefinite matrices in $\S^m$.
The integral in   (\ref{eq_logpartition}) is the standard multivariate Gaussian integral.
We find
$$
A(\theta) \,\,=\,\,-\frac{1}{2}\log\det(\theta) + \frac{m}{2}\log(2\pi).
$$
Hence the multivariate Gaussian distribution on $\R^m$ forms a  rational exponential family.
Its defining polynomial $f(\theta)$ is the determinant of  a symmetric $m \times m$-matrix, and
we have $\alpha = 1/2$.
The gradient map $F$ in (\ref{eq:ratlmap1})  is $\,\theta \mapsto \frac{1}{2}\theta^{-1}$.
Theorem \ref{thm:brown} says that matrix inversion is an involution on the open
cone $C = {\rm int}(K)$ of positive definite matrices in $\S^{m}$.
\hfill $\diamondsuit$
\end{example}

\begin{example}[({\em Full discrete family})]\label{sub:Full_discrete} \rm
We consider discrete random variables with finite state space
$\cX=\{1,2,\ldots , d\}$. Here $\nu$ is the counting measure on $\cX$, and the
sufficient statistics are the unit vectors $T(i)={\bf e}_i$ in $\R^d$.
The probability distribution function (\ref{eq:expfam}) has the form
$$  p_\theta(i) \,\, = \,\,  {\rm exp} \bigl( -\theta^T \cdot T(i) - A(\theta) \bigr) \,\, = \,\,
\frac{e^{-{\theta_i}}}{e^{A(\theta)}}. $$
 The log-partition function equals
 \begin{equation}
 \label{eq:FDF1}
 A(\theta) \,\,=\,\,\log\bigl( e^{-\theta_1} + e^{-\theta_2} + \cdots + e^{-\theta_d}  \bigr).
 \end{equation}
 The gradient map $F$ satisfies
\begin{equation}
 \label{eq:FDF2}
 F(\theta) \,\, = \,\,
-\nabla A(\theta) \,=\,
\frac{1}{\sum_{i=1}^{d} e^{-\theta_i}} \cdot
\left(e^{-\theta_1},e^{-\theta_2},\ldots ,e^{-\theta_{d}}\right).
\end{equation}
The convex set in (\ref{eq:defC}) is the entire space, $C=\mathbb{R}^{d}$,
so this family is regular.
The space of sufficient statistics, defined in  (\ref{eq:defK}),  is the closed
$(d-1)$-dimensional probability simplex
$$ K \,\,=\,\, {\rm conv}({\bf e}_1,{\bf e}_2,\ldots,{\bf e}_d) \,\,= \,\,
\bigl\{ \sigma \in \R^d_{\geq 0} \,:\, \sigma_1 + \sigma_2 + \cdots + \sigma_d = 1 \bigr\}.$$
The gradient map takes $C$ to the interior of $K$. But
this is not a bijection. Theorem \ref{thm:brown} does not hold here
because this exponential family is not minimal.
Following  \cite[Example 1.3]{brown}, we can turn this into a
{\em minimal} exponential family by replacing $d$ with $d-1$ and
projecting $T(x)$ on the first $d-1$ coordinates. Now $C = \R^{d-1}$,
 the formulas (\ref{eq:FDF1}) and (\ref{eq:FDF2}) hold with $\theta_d = 0$,
and the gradient map (\ref{eq:FDF2}) gives a bijection
between $C$ and the open simplex ${\rm int}(K)$.

The gradient map is analytic but not algebraic, so this
is not a rational exponential family.
This can be fixed by replacing the
parameter space $C = \R^{d-1}$ with the open orthant
$\R^{d-1}_{> 0}$, by way of the bijection
$r: \theta \mapsto (e^{-\theta_1},\dots,e^{-\theta_{d-1}})$.
The gradient map  is the composition of $r$ and a rational function.
This will allow us in  Example \ref{ex:rationaldiscrete1}
to obtain {\em discrete exponential families} from lattices $\mathcal{L}$.
 As noted in \cite{DS,HS,RKA} and elsewhere,
these are precisely {\em toric varieties}. \hfill $\diamondsuit$
\end{example}

We close this section by pointing out a connection between statistics and optimization
which seems to have not yet been stated explicitly  in the mathematical literature.

\begin{remark} \rm
In his 1997 paper  \cite{guler1}, G\"uler presented a
framework for {\em interior point methods} in convex optimization
based on canonical barrier functions. This theory can be viewed as a special case of exponential
family theory. The following discussion  explains how this works.

Fix a $d$-dimensional pointed convex cone $K$ in $\R^d$, and
let $\nu$ be the Lebesgue measure on $K$.
To match the notation above, set $\cX = K$ and let $T$ be the identity on $K$.
Then, in the setting of  \cite[Section 3]{guler1},
the log-partition function $A(\theta)$ in (\ref{eq_logpartition}) becomes
   the logarithm of the \emph{characteristic function}
 of the dual cone  $K^{\vee} = \overline{C}$.  This is, up to an additive constant term $\log d!$, the  {\em universal barrier function} of~$C$:
\begin{equation}
\label{eq:Athetaintegral}
A(\theta)\,\,=\,\,\,\log \int_{K}\exp(-\langle \theta,\sigma\rangle)\dd \sigma.
\end{equation}
For example, let $C$ be the
positive definite cone ${\rm PD}_{m}$. In G\"uler's theory \cite{guler1,guler2},
 the integral~(\ref{eq:Athetaintegral})  corresponds to interior point methods in
 {\em semidefinite programming} and the induced rational exponential 
 family is precisely that in
 Example \ref{ex:fullgaussian}. In other words, the
multivariate Gaussian family comes from
the universal barrier function of $\textrm{PD}_m$.


 Our gradient map $F$ is known as the
{\em duality mapping} in convex optimization \cite[Section 5]{guler1}.
That it defines a bijection between the interiors of $C$ and $K$
is shown in \cite[Theorem 5.1]{guler1}. This is a special case of
what statisticians know about
exponential families  (Theorem~\ref{thm:brown}).
\end{remark}

The optimization problem of most interest in statistics is that of finding the maximum likelihood estimator (MLE). For an exponential family the solution to this problem can be described as follows.
Suppose we are given $n$ independent observations
$x_1,\ldots,x_n$ in $\cX$. We record these observations by the average value
of their sufficient statistics. This is the~point
$$ \hat\sigma \, = \, \frac{1}{n} \bigl( T(x_1) + T(x_2) + \cdots + T(x_n)\bigr) \,\,\, \in \,\, K . $$
The MLE is defined as the canonical parameter
vector $\hat \theta$ that makes the data $x_1,\ldots,x_n$ most likely.
By \cite[Theorem 5.5]{brown},  the MLE exists and is unique
precisely when $\hat\sigma$ lies in the interior of~$K$. If this holds, then
$\hat \theta$ is the unique inverse image of $\hat\sigma$ in $C$
under the gradient map.

In other words, computing the MLE for an exponential family means intersecting the
fiber $F^{-1}(\hat \sigma)$ 
of the gradient map $F$ with the convex set $C$.
For a rational exponential family we can study this fiber
and its degree 
using methods from algebraic geometry.
This was pioneered in~\cite{stuhl} for the special case of
linear Gaussian concentration models. It will be~developed~in~full generality in this paper. The pertinent algebraic complexity measures,
namely the gradient multidegree and the maximum likelihood degree
(ML degree), will be introduced in Section~\ref{sec:ml_degree}.

\section{Hyperbolic Polynomials and Riesz Kernels}
\label{sec:dis_gau_hyp}

In this section, we introduce a new class of rational exponential families associated with hyperbolic polynomials and their
 hyperbolicity cones.
Recall (e.g.~from \cite{guler2,KPV}) that a  homogeneous polynomial $f $ of degree $p$ in $\theta=(\theta_1,\ldots,\theta_d)$
is {\em hyperbolic} if there exists a point $\tau$ in
$\R^d \backslash \{f=0\}$ such that every line
through $\tau$ meets the hypersurface $\{f=0\}$ in $p$ real points, counting multiplicities.
Assume that this holds and let $C$ be the connected component
of $\R^d \backslash \{f=0\}$ that contains $\tau$.
It is known that $C$ is an open convex cone in $\R^d$.
This cone is called the {\em hyperbolicity cone} of~$f$.
As seen in Example  \ref{ex_prodlinform}, it may depend on $\tau$.

\begin{example}
\label{ex_Gaussian_hyp} \rm
The standard example of a hyperbolic polynomial is the determinant of a symmetric
$m \times m$-matrix of unknowns. Here $d = \binom{m+1}{2}$, the hyperbolicity
cone $C$ consists of the positive definite matrices in $\mathbb{S}^m$, where the witness point $\tau$ can be taken to be the identity matrix. Hence, we
 recover the Gaussian family in Example \ref{ex:fullgaussian}.~$\diamondsuit$
\end{example}

\begin{example}
\label{ex_prodlinform}
\rm
Let $C$ be a convex polyhedral cone in $\R^d$, defined by
$s$ linear inequalities  $\ell_1(\theta) > 0, \ldots,  \ell_s(\theta) > 0$.
The product of linear forms $f(\theta) = \prod_{i=1}^s \ell_i(\theta)$ is a
hyperbolic polynomial. The hyperbolicity cone $C$
can here be replaced with any of the other chambers in the arrangement
of hyperplanes $\{\ell_j = 0\}_{j=1,\ldots,s}$.
In the setting of \cite{guler1,guler2}, this models interior point methods for {\em linear programming}, with barrier function
$A(\theta) = -\sum_{j=1}^s {\rm log}(\ell_j(\theta))$.
Its gradient map
defines a bijection between  $C$
and the interior of the dual cone $K = C^\vee$.
\hfill $\diamondsuit$
\end{example}

The following recent result by Scott and Sokal~\cite{sokal} motivates studying hyperbolic polynomials in the context of exponential families. It shows that any homogeneous polynomial that is (up to a negative power) the partition function of an exponential family must be hyperbolic.

\begin{theorem}
\label{th:gardingconverse}
Let $f$ be a homogeneous polynomial in $\R[\theta_{1},\ldots, \theta_{d}]$ that is strictly positive on an open convex cone $C$. If there exists a positive number $\alpha$ and a measure $\nu$ such that
$$
f(\theta)^{-\alpha}=\int_{C^{\vee}}\exp(-\langle \theta,\sigma\rangle)\;\nu(\dd \sigma) \qquad \hbox{for all $\theta\in C$,}
$$
then $f$ is hyperbolic with respect to each point in the cone $C$.
\end{theorem}

\begin{proof}
A real analytic function $g$ is {\em completely monotone} on a convex cone $C$ if and only~if
\begin{equation}\label{def:cm}
(-1)^k\frac{\partial^k g}{\partial_{ u_1}\cdots\partial_{u_k}}\,\geq\, 0 \,\,\,\text{ on }C \qquad
\hbox{ for all $k\in \N$ and $u_1,\dots,u_k\in C$.}
\end{equation}
There is a beautiful characterization of completely monotone functions, known as 
the Bernstein-Hausdorff-Widder-Choquet Theorem \cite{Choquet,sokal}. Namely, an analytic function $g$ defined on $C$ is completely monotone if and only if there exists a  measure $\mu$ on the dual cone $C^\vee$ such that:
$$g(\theta)\,\,=\,\,\int_{C^{\vee}}\exp(-\langle \theta,\sigma\rangle)\;\mu(\dd \sigma).$$

By our assumptions it follows that
the function $f(\theta)^{-\alpha}$ is completely monotone on
the convex cone~$C$. Then, by \cite[Corollary 2.3]{sokal}, we
can conclude that~$f$ has no zeros in the subset~$C+i\R^{d}$ of~$\C^d$.
This means that the polynomial~$f$ is hyperbolic with respect to each point~$\tau$ in $C$.
\end{proof}

The following theorem, essentially due to  G{\aa}rding~\cite{garding}, suggests a construction of exponential families from complete hyperbolic polynomials.
A hyperbolic polynomial is called {\em complete} if
the hyperbolicity cone $C$ is pointed, that is, $ \,C \,\cap \,(-C) = \{{\bf 0}\}$ in $\R^d$. 
  In what follows, both $\dd \eta$ and $\dd \sigma$ denote the Lebesgue measure on $\R^d$. The appearance of the non-real scalar
 $i = \sqrt{-1}$ may look puzzling at first, but it is essential for G{\aa}rding's  theory.
 
\begin{theorem}\label{th:garding}
Let $f$ be a complete hyperbolic polynomial
with hyperbolicity cone $\,C \subset \R^d$, and  fix $\,\alpha > d$.
The following integral converges for any $\theta \in $C,
 its value does not depend on the choice of $\theta$,
 and it is supported on  $K = C^\vee$:
\begin{equation}
\label{eq:intrep}
q_{\alpha}(\sigma)\,\,=\,\,(2\pi)^{-d}\int_{\R^{d}} f(\theta+i\eta)^{-\alpha} \exp(\langle\theta+i \eta, \sigma\rangle)\dd \eta.
\end{equation}
The polynomial $f$ can be recovered via the formula
\begin{equation}
\label{eq:intrep2}
\begin{matrix} \,\,\,
 f(\theta)^{-\alpha}&=&\int_{K}\exp(-\langle \theta,\sigma\rangle)q_{\alpha}(\sigma)\dd \sigma \qquad
\hbox{for all $\,\theta\in C$.}
\end{matrix}
\end{equation}
\end{theorem}

\begin{proof} It follows directly from Theorem~3.1 in G{\aa}rding's paper~\cite{garding} that
the formula $f(\theta)^{-\alpha}=\int_{K}\exp(-\langle \theta,\sigma\rangle)
q_{\alpha}(\sigma)\dd \sigma\,$
holds for all $\theta \in C$, that
 $q_{\alpha}(\sigma)$ does not depend on the choice $\theta\in C$,
and that  $q_{\alpha}(\sigma)$
vanishes for $\sigma \notin K$. In general,
the expression for $q_{\alpha}(\sigma)$ may not be a well-defined function. However, it is well-defined if $\alpha>d$. See also~\cite[Theorem 6.4]{guler2}. As we shall
see later, the condition $\alpha>d$ is only sufficient but not necessary.
\end{proof}

The function $q_\alpha(\sigma)$ in Theorem~\ref{th:garding} is known
as the \emph{Riesz kernel}. It is uniquely determined (up to a set
with Lebesgue measure zero) by the integral representation (\ref{eq:intrep}). If $q_{\alpha}\geq 0$, then $\nu_{\alpha}(\dd \sigma) = q_{\alpha}(\sigma)\dd \sigma$ defines a measure on $K$ and allows the construction of \emph{hyperbolic exponential families}. Various authors suggested that $q_{\alpha}$ is  nonnegative, see e.g. \cite[p 371]{guler2}. However, we are not aware of any proof of this fact and we therefore state this as a conjecture:

\begin{conjecture}\label{conj_Riesz} 
 For $\alpha \gg 0$
 the Riesz kernel  $q_\alpha(\sigma)$ in (\ref{eq:intrep})
takes nonnegative values on~$K$.
\end{conjecture}

In the following we slightly modify the terminology
relative to what is standard in statistics. Namely, we
  define hyperbolic exponential families with respect to a possibly signed measure:

\begin{definition}[({\em Hyperbolic exponential families})] \rm
Let $f \in \R[\theta_1,\ldots,\theta_d]$ be a complete hyperbolic polynomial with hyperbolicity cone $\,C \subset \R^d$. The corresponding  {\em hyperbolic exponential family} $(\mathcal{X},\nu,T)$ is a rational exponential family with canonical parameter space $C$, sample space $\cX = C^\vee$, sufficient statistics $T(x)=x$, and $\nu$ given by the Riesz kernel~(\ref{eq:intrep}). 
\end{definition}


\begin{remark}\label{rem_Riesz}
It is known from optimization theory \cite{Renegar} that the bijection between $C$ and the interior of $K= C^{\vee}$ described in Theorem~\ref{thm:brown} holds for complete hyperbolic 
polynomials $f$ with hyperbolicity cone $C \subset \R^d$. 
Namely, the gradient function $F(\theta)=-\log f(\theta)$ defines a bijection
 between $C$ and the interior of $K=C^{\vee}$. This follows from the fact that 
 $F(\theta)$ is a self-concordant barrier function of $C$ and does not require 
 non-negativity of the Riesz kernel; see \cite[Section 6]{guler2} for details. 
 This justifies our slight misuse of vocabulary in the above definition.
\end{remark}

Our introduction of hyperbolic exponential families 
appears to be a novel contribution to statistics. In order for 
these to become  useful for  data analysis,
it is essential to gain a better understanding of
the Riesz kernel. In general, it is a hard problem
to compute  $q_{\alpha}(\sigma)$ in an explicit form.
In the remainder of this section, we discuss some
concrete examples, starting with the Gaussian family in Example~\ref{ex_Gaussian_hyp}.
 These are situations where the
hyperbolic polynomial $f$ has a symmetric determinantal representation. In this case, the kernel $q_\alpha(\sigma)$ is closely related
to the {\em Wishart distribution} on the positive
definite cone~${\rm PD}_m$, i.e.~the sampling distribution for covariance matrices in the
Gaussian family (see the proof of Proposition~\ref{prop_Wishart} below). We refer to the article of Scott and Sokal \cite{sokal} for further examples,
many results, and open problems.

\begin{proposition}
\label{prop_Wishart}
Let $f$ be the determinant of a symmetric $m \times m$-matrix $\theta$ and
$C = {\rm PD}_m$ the cone of positive definite matrices.
For $\alpha>\frac{m-1}{2}$, the corresponding Riesz kernel equals
\begin{equation}
\label{eq:qGamma}
q_{\alpha}(\sigma)  \,\,= \,\,\frac{1}{\Gamma_{m}(\alpha)}\,\det(\sigma)^{\alpha-\frac{m+1}{2}}.
\end{equation}
\end{proposition}

The denominator $\Gamma_{m}\left(\alpha\right)$ in equation (\ref{eq:qGamma})
denotes the {\em multivariate gamma function}
$$
\Gamma_{m}\left(\alpha\right)\,\,:=\,\,\,\pi^{m(m-1)/4} \cdot \prod_{j=1}^{m}\,\Gamma\left(\alpha-\frac{1-j}{2}\right).
$$

\begin{proof}
We compute the Riesz kernel from the Wishart distribution.
This is a probability distribution over the positive definite cone ${\rm PD}_m$ with two parameters: the degrees of freedom $n\in (m-1,\infty)$ and a scale matrix $\Lambda\in {\rm PD}_m$. The density function of the Wishart distribution~is
\begin{equation}\label{eq:wishart}
\frac{1}{2^{mn/2}\,\det(\Lambda)^{n/2}\,\Gamma_{m}(n/2)}\det(\sigma)^{\frac{n-m-1}{2}}\exp(-\frac{1}{2}{\rm tr}(\Lambda^{-1}\sigma)).
\end{equation}

We replace the parameters $n$ and $\Lambda$ with
the new parameters $\alpha=n/2$ and $\theta=(2\Lambda)^{-1}$,
and we use the fact that the integral  of (\ref{eq:wishart}) over $\sigma \in {\rm PD}_m$ is equal to $1$.
This implies
\begin{equation}\label{eq:aux_riesz}
\det(\theta)^{-\alpha} \,\,=\, \, \,
\frac{1}{\Gamma_{m}(\alpha)} \int_{{\rm PD}_m} \!\!
\det(\sigma)^{\alpha-\frac{m+1}{2}}\exp(-\langle\theta, \sigma\rangle)\,\dd \sigma
\qquad \hbox{for $\alpha>\frac{m-1}{2}$},
\end{equation}
 where $\langle\,\cdot\,,\,\cdot\,\rangle$ denotes the trace inner product
on symmetric matrices.
 By comparing this expression with (\ref{eq:intrep2}) and by using the uniqueness of this integral representation, we see that the
Riesz kernel of $f = {\rm det}(\theta)$ for $\alpha>\frac{m-1}{2}$ is given by
the right hand side of (\ref{eq:qGamma}).
\end{proof}

Using an analog of Proposition~\ref{prop_Wishart} for Hermitian matrices, one can compute the Riesz
kernel of the elementary symmetric polynomial $E_2(\theta)$.
This was carried out in \cite[Corollary 5.8]{sokal}. We here go over this for the
smallest instance. This is also featured in \cite[Proposition~3.8]{sokal}.

\begin{example}
\label{ex:RieszE2}
 \rm
Consider $f = \theta_1\theta_2 +\theta_1\theta_3+\theta_2\theta_3$ with hyperbolicity
cone $C \subset \R^3$. Note that
$$ f(\theta) \,=\, \det \begin{pmatrix}\theta_1 +\theta_3 & \theta_3 \\ \theta_3 & \theta_2+\theta_3
 \end{pmatrix}. $$
By Proposition~\ref{prop_Wishart}, the following integral representation
is valid for $\alpha>1/2$:
$$ f(\theta)^{-\alpha}\,=\,
\frac{1}{\Gamma_{2}(\alpha)} \int_{{\rm PD}_2} \exp\big(\! -(\theta_{1}+\theta_3)\sigma_{11}
- (\theta_2+\theta_3)\sigma_{22} - 2\theta_3\sigma_{12}\big)\,(\sigma_{11}\sigma_{22}-\sigma_{12}^2)^{\alpha-3/2}\,\dd \sigma .$$
 By setting $\sigma_1 = \sigma_{11}$, $\sigma_2 = \sigma_{22}$, and $\sigma_3 = \sigma_{11} +\sigma_{22}+2\sigma_{12}$, this transforms into an integral  over the dual
 $K = C^\vee$ to the hyperbolicity cone. Namely, we get the integral representation
$$f(\theta)^{-\alpha} = \frac{2^{2-2\alpha}}{\Gamma_{2}(\alpha)} \int_{K} \!\! 
\exp\big({-}\theta_{1}\sigma_1 {-}\theta_2\sigma_{2} {-} \theta_3\sigma_3\big)\,\big(2(\sigma_1\sigma_2{+}\sigma_1\sigma_3{+}\sigma_2\sigma_3) 
-(\sigma_1^2{+}\sigma_2^2{+}\sigma_3^2)\big)^{\alpha-3/2}\,\dd \sigma.$$
By comparing this expression with (\ref{eq:intrep2}), we see that the
Riesz kernel of $f $ is equal to
$$ q_\alpha(\sigma) \,\, = \,\,\,
\frac{2^{2-2\alpha}}{\Gamma_{2}(\alpha)} \;\big(2(\sigma_1\sigma_2+\sigma_1\sigma_3+\sigma_2\sigma_3) -(\sigma_1^2+\sigma_2^2+\sigma_3^2)\big)^{\alpha-3/2}.$$
Note that the polynomial in the parenthesis vanishes on the boundary of $K$;
it is the quadric dual to $f$.
See \cite[equation (4.16)]{sokal} for the analogous formula for $E_2(\theta)$
in four variables.
\hfill $\diamondsuit$
\end{example}

We next examine $f(\theta) = \theta_1 \theta_2 \cdots \theta_m$.
In the context of Proposition~\ref{prop_Wishart},
this hyperbolic polynomial is  the determinant of a diagonal matrix.
Its hyperbolicity cone is the positive orthant $C =  \mathbb{R}^m_{>0}$.
This cone is self-dual, so that $K = \mathbb{R}^m_{\geq 0}$. We consider the
case $\alpha = 1/2$.

\begin{proposition}
\label{prop_Diagonal}
The Riesz kernel for determinants of $m \times m$ diagonal matrices equals
\begin{equation}
\label{eq:diagonalriesz}
q_{1/2}(\sigma) \,\,=\,\, \pi^{-m/2}\,(\sigma_1 \sigma_2 \cdots \sigma_m)^{-1/2}.
\end{equation}
\end{proposition}

\begin{proof}
 Let $\theta$ denote the diagonal matrix $\textrm{diag}(\theta_1,\dots ,\theta_m)$.
 We use the fact that the  Gaussian density on $\R^m$ with mean ${\bf 0}$
 and inverse covariance matrix $\theta$ integrates to $1$.
 This implies
$$
\det(\theta)^{-1/2}
\,\,\,=\,\,\,\frac{1}{(2\pi)^{m/2}} \int_{\mathbb{R}^m} \exp\Big(-\frac{1}{2}\, x^T \theta \!\;x\Big)\,\dd x
\,\,\,=\,\,\,
(2/\pi)^{m/2} \int_{\mathbb{R}^m_{\geq 0}} \! \exp\Big(-\frac{1}{2}\,\sum_{i=1}^m \theta_i\!\; x_i^2\Big)\,\dd x.
$$
By making the change of variables $\frac{1}{2}x_i^2\to \sigma_i$, we obtain
\begin{equation}
\label{eq_Gaussian_diag}
\det(\theta)^{-1/2} \,\,=\,\,
\pi^{-m/2} \int_{\mathbb{R}^m_{\geq 0}} \!\!
  \exp\Big(-\sum_{i=1}^m \theta_i \sigma_i\Big)\,(\sigma_1\cdots \sigma_m)^{-1/2}\, \dd \sigma.
\end{equation}
By comparing this expression with (\ref{eq:intrep2}), we see that the
Riesz kernel equals (\ref{eq:diagonalriesz}).
\end{proof}

In principle,  the computation of Riesz kernels as in
Propositions~\ref{prop_Wishart} and \ref{prop_Diagonal}
can be restricted to  cases when the parameters satisfy some linear constraints.
The resulting hyperbolic polynomials $f(\theta)$  admit a
symmetric determinantal representation.
This requires integrating out the complementary parameters,
a task that is very difficult to do. Even the case of products of linear forms
(Example \ref {ex_prodlinform}) is challenging, as the following example shows.

\begin{example}
 \label{ex_diag_Gaussian} \rm
Let $f(\theta)=\theta_{1}\theta_{2}(\theta_{1}+\theta_{2})(\theta_{1}-\theta_{2})$,
and take the hyperbolicity cone to be $C=\{\theta_{1}>\theta_{2}>0\}$. Then its dual is $K=\{ \sigma_{1}\geq 0,\sigma_{1}+\sigma_{2}\geq 0\}$. By applying (\ref{eq_Gaussian_diag}) we obtain
$$ \det(\theta)^{-1/2} = \pi^{-2} \int_{\mathbb{R}^4_{\geq 0}} \exp\big(-\theta_1 y_1 -\theta_2y_2 - (\theta_1+\theta_2) y_3 - (\theta_1-\theta_2) y_4\big)\,\big(y_1y_2y_3y_4\big)^{-1/2}\, \dd y.$$
By replacing $\sigma_1=y_1+y_3+y_4$ and $\sigma_2=y_2+y_3-y_4$, and defining
$$K' = \Big\{(y_3,y_4)\in\mathbb{R}^2_{\geq 0} \;\mid \; y_3 + y_4\leq \sigma_1
\,\,\hbox{and}\,\, y_3-y_4\leq \sigma_2 \Big\},$$ we find
$$ \det(\theta)^{-1/2} \,= \, \pi^{-2}\!\!
 \int_{K} \exp\big(-\theta_1 \sigma_1 -\theta_2\sigma_2\big)\!\!\left(\int_{K'}\!\big((\sigma_1-y_3-y_4)(\sigma_2-y_3+y_4)y_3y_4\big)^{-1/2}\, \dd y \right)  \dd \sigma .$$
 Hence the Riesz kernel is
$$q_{1/2}(\sigma)  \,= \,  \pi^{-2} \int_{K'}\!\big((\sigma_1-y_3-y_4)(\sigma_2-y_3+y_4)y_3y_4\big)^{-1/2}\, \dd y.$$
 The integral over $K'$ can be expressed in terms of generalized hypergeometric functions. Performing this computation requires a case distinction, and we here consider the case when $\sigma_2\geq \sigma_1$. Then we may substitute $y_3\to \sigma_1 v$ and $y_4\to \sigma_1 u$ and obtain
 \begin{align*}
\int_{K'}\!\big((\sigma_1-y_3-y_4)&(\sigma_2-y_3+y_4)y_3y_4\big)^{-1/2}\, \dd y  \; 
\\
 & =\, \sigma_1^{\frac{1}{2}}\sigma_2^{-\frac{1}{2}}\int_{0}^1\!\!\int_0^{1-v} \!\!\!\! u^{-1/2} v^{-1/2} \big(1-u-v\big)^{-1/2} \Big(1+\frac{\sigma_1}{\sigma_2}u -\frac{\sigma_1}{\sigma_2}v\Big)^{-1/2}\, \dd u \, \dd v\\
 & = \,2\pi\,\sigma_1^{\frac{1}{2}}\sigma_2^{-\frac{1}{2}} \;F_1\Big(\frac{1}{2};\, \frac{1}{2}, \frac{1}{2};\, \frac{3}{2};\, -\frac{\sigma_1}{\sigma_2}, \frac{\sigma_1}{\sigma_2}\Big),
 \end{align*}
where $F_1$ is the first {\em Appell hypergeometric function} in two unknowns.
Here we are using
formula (6) in {\tt http://mathworld.wolfram.com/AppellHypergeometricFunction.html}. \\
A similar analysis can be carried out for the other case  $\sigma_1> \sigma_2$,
but this is omitted here.
\hfill $\diamondsuit$
\end{example}

Suppose now that $f(\theta)$ is an arbitrary product of linear forms
as in Example \ref {ex_prodlinform}, and we wish to
find integral representations of $f(\theta)^{-\alpha}$.
The Riesz kernel $q_\alpha(\sigma)$ is expressible
in terms of Aomoto-Gel'fand hypergeometric functions~\cite{Aomoto,Vasilev}.
The corresponding hypergeometric D-modules
should be useful for representing Riesz kernels.
More generally, it would be interesting to
develop the {\em holonomic gradient method}
\cite{HNTT} for  hyperbolic exponential families.

\section{Restricting to Linear Subspaces}
\label{sec:lin_sub}

Many important statistical models arise by
restricting a given exponential family $(\cX,\nu,T)$ to
a linear subspace $\mathcal{L} \subset \R^d$ of the canonical parameters.
  Figure~\ref{fig:yellowgreen} is meant to illustrate this.
We shall assume $C\cap \mathcal L\neq \emptyset$ and  $\dim \mathcal L=c$. Fix a basis of $\mathcal{L}$ and let $L\in\mathbb{R}^{c\times d}$ denote a matrix whose rows are the elements of that basis.
The map $\pi_\mathcal{L} : \R^d \rightarrow \R^d/\mathcal{L}^\perp \simeq \R^c$
is given in coordinates by the matrix $L$. In what follows we identify
$L=\pi_\mathcal{L}$. Later on, we often abuse notation and we write
$\mathcal{L}$ also for the projectivization of the
linear subspace $\mathcal{L} \subset \R^d$.
In that context, $\mathcal{L}$ will be a
plane of dimension $c-1$ in $\R \PP^{d-1}$ or $\C \PP^{d-1}$. Note that all results in this section hold when $\nu$ is a signed measure and hence apply in particular to hyperbolic exponential families.

For every $\theta\in C\cap {\mathcal L}$ there exists a
vector $\tau\in\mathbb{R}^c$ such that
$\theta = L^T \tau$, and hence
\begin{equation}\label{eq:projNconstr}
\langle\theta, T(x)\rangle=\langle L^{T}\tau, T(x)\rangle = \langle \tau, (L\circ T)(x)\rangle.
\end{equation}
By restricting $C$ to $C\cap \mathcal L$ we obtain an exponential family $(\cX,\nu,T_{\mathcal L})$ with sample space $\mathcal{X}$, canonical parameter space $C_{\mathcal{L}} = C\cap \mathcal{L}$, sufficient statistics $T_{\mathcal L}(x) =
(L\circ T)(x)$, space of sufficient statistics $K_{\mathcal{L}}={\rm conv}(T_{\mathcal{L}}(\mathcal{X}))$, log-partition function $A_\mathcal{L}$, and gradient map $F_{\mathcal{L}}$.

\begin{lemma} The exponential family $(\cX,\nu,T_{\mathcal L})$ is invariant under the choice
of  basis for~$\mathcal{L}$.
\end{lemma}

\begin{proof} Let $M\in\mathbb{R}^{c\times d}$ denote another basis for $\mathcal L$. There exists an invertible matrix $U\in \R^{c\times c}$ such that $L=UM$. Let $\eta=U^{T}\tau$ be another coordinate system on $\mathcal L$. Then $L^{T}\tau=M^{T}\eta$ and, by (\ref{eq:projNconstr}), for every $\theta\in C\cap \mathcal L$ we find $\langle \tau,
(L \circ T)(x)\rangle = \langle \theta, T(x)\rangle = \langle \eta, (M \circ T)(x)\rangle.$
\end{proof}

In what follows we fix the basis consisting of the rows of $L$,
and we identify $\R^c$ with $\mathcal{L}$ via $\tau \mapsto L^T \tau$.
This means that the convex set $C_\mathcal{L} = C \cap \mathcal{L}$ lives in $\R^c$
with coordinates $\tau$.

\begin{lemma}The gradient map of the restricted exponential family $(\cX,\nu,T_{\mathcal L})$ is given by
\begin{equation}\label{eq:FLtheta}
F_{\mathcal L}\;=\; L\circ F\circ L^{T}.
\end{equation}
\end{lemma}

\begin{proof}
A computation similar to  (\ref{eq:projNconstr}) shows  that
 all canonical parameters $\tau\in C_{\mathcal L}$ satisfy
\begin{equation}\label{eq:ALtheta}
A_{\mathcal L}(\tau)\,=\,
\int_{\cX}\exp(-\langle \tau, (L \circ T)(x)\rangle )\nu(\dd x)
\,=\,\int_{\cX}\exp(-\langle L^{T}\tau, T(x)\rangle)\nu(\dd x)\,=\,A(L^{T}\tau).
\end{equation}
Thus the derivative of $A_{\mathcal L}(\tau)$ with respect to $\tau_{i}$ is equal to the
directional derivative of $A(\theta)$ in the direction given by the $i$-th row of $L$. This implies the
following identify which proves~(\ref{eq:FLtheta}):
$$\nabla_{\!\tau} A_{\mathcal L}(\tau) \,\,\,= \,\,\, L\nabla_{\!\theta} A(\theta)\Big|_{\theta=L^{T}\tau}. $$
\vskip -0.8cm
\end{proof}

\begin{example}[({\em Toric varieties from discrete exponential families})]
\label{ex:rationaldiscrete1} \rm
We build on Example~\ref{sub:Full_discrete} and consider linear restrictions of the parameter space
in the full discrete exponential family. Let $\mathcal{L}$ be a linear subspace of $\R^d$ that
is defined over the rational numbers, and take $L \in \mathbb{Z}^{c \times d}$. We
assume that $(1,1,\dots,1)\in \mathcal L$. Here $C=\R^{d}$ and thus $C\cap \mathcal L=\mathcal L$.
The set $K_\mathcal{L}$ is the image of the $(d-1)$-dimensional simplex $K$
under the linear map $L: \R^d \rightarrow \R^c$. Thus $K_\mathcal{L}$ is the
$(c-1)$-dimensional convex polytope obtained as the convex hull of the columns of $L$.

Consider the image $F(\mathcal{L})$ of $\mathcal{L}$ under the gradient map
$F$ in (\ref{eq:FDF2}). This image is a semialgebraic set of dimension $c-1$ inside the
$(d-1)$-dimensional simplex $K$.  To be explicit,
$$ F(\mathcal{L}) \quad = \quad \bigl\{\,(\sigma_1,\sigma_2,\ldots,\sigma_d) \in K \,:\,
\prod_{i=1}^d \sigma_i^{u_i} = 1 \,\,\,\hbox{for all} \,\, u \in \mathcal{L}^\perp\, \bigr\}. $$
Since $L$ has integer entries, we can here
replace the space $\mathcal{L}^\perp $
with the lattice $\mathcal{L}^\perp_\mathbb{Z} = \mathcal{L}^\perp \cap \mathbb{Z}^d$.

We write $\mathcal L^{F}$ for the Zariski closure of $F(\mathcal L)$ in $\C \PP^{d-1}$,
regarded as the complexification of $K$.
This is the {\em projective toric variety} associated with the polytope $K_\mathcal{L}$.
To be explicit,
$$ \mathcal{L}^F \quad = \quad
\bigl\{  \, (\sigma_1:\sigma_2: \ldots: \sigma_d) \in \C \PP^{d-1} \,:\,
\prod_{i: u_i > 0} \sigma_i^{u_i} \,=
\prod_{j: u_j < 0} \sigma_j^{-u_j} \!
 \,\,\,\,\hbox{for all} \,\, u \in \mathcal{L}_\mathbb{Z}^\perp\, \bigr\}.
$$
In algebraic statistics, one refers to $F(\mathcal{L})$ as the {\em toric model}
of the lattice $\mathcal{L}_\mathbb{Z}$.
The linear map $ L$ takes this model
bijectively onto the polytope $K_\mathcal{L}$ of sufficient statistics.
In geometry, this bijection is known as the {\em moment map}
of the toric variety $\mathcal{L}^F$. Given any  point in ${\rm int}(K_\mathcal{L})$,
its unique preimage in $F(\mathcal{L})$ is known as the {\em Birch point} or  the {\em MLE}.
For further reading on toric models see
 \cite[Section 1.2.2]{ASCB}, \cite{RKA},
and  other introductions to algebraic statistics. \hfill $\diamondsuit$
\end{example}

\begin{figure}[h]
\centering
 \includegraphics[width=6.8cm]{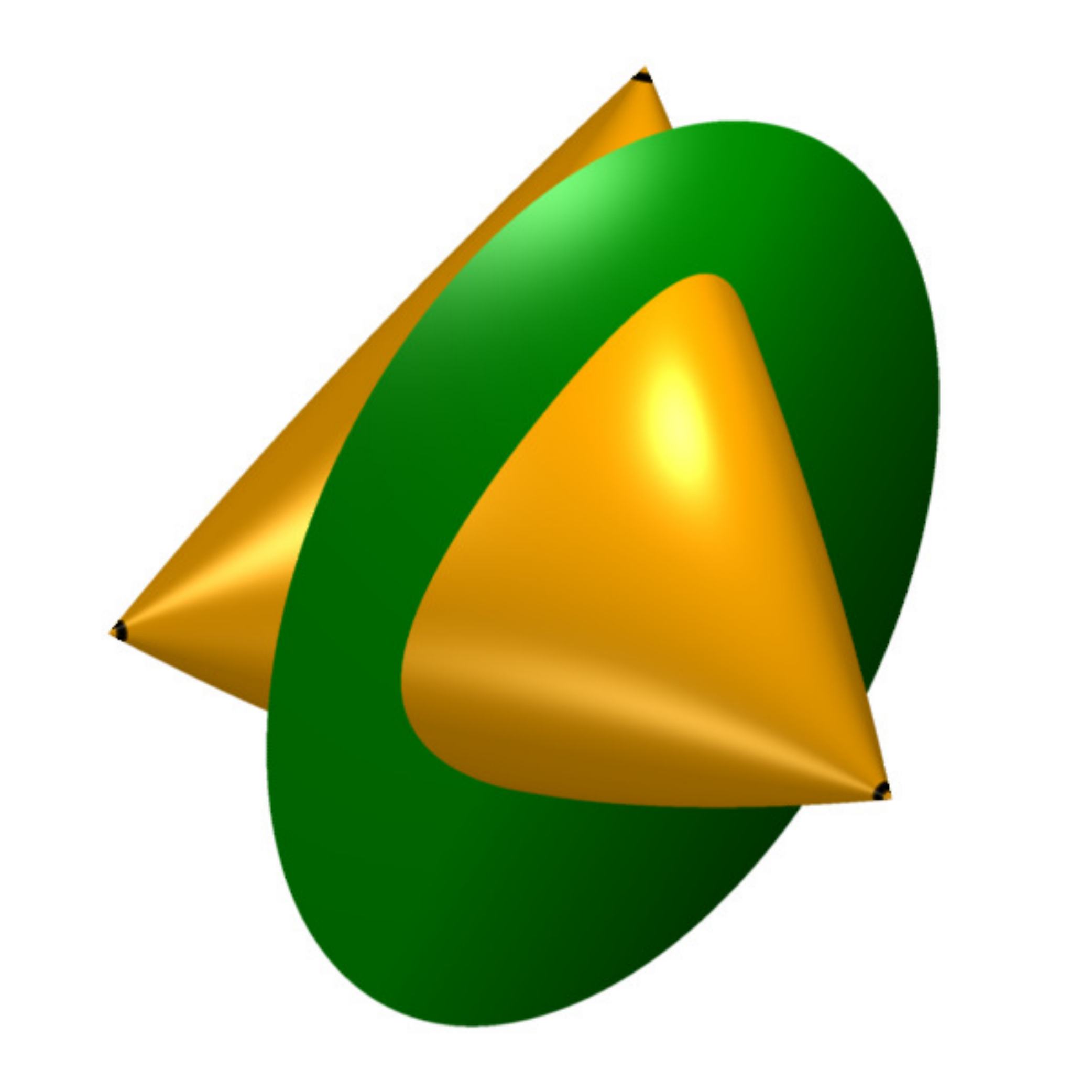}
  \qquad \includegraphics[width=6.1cm]{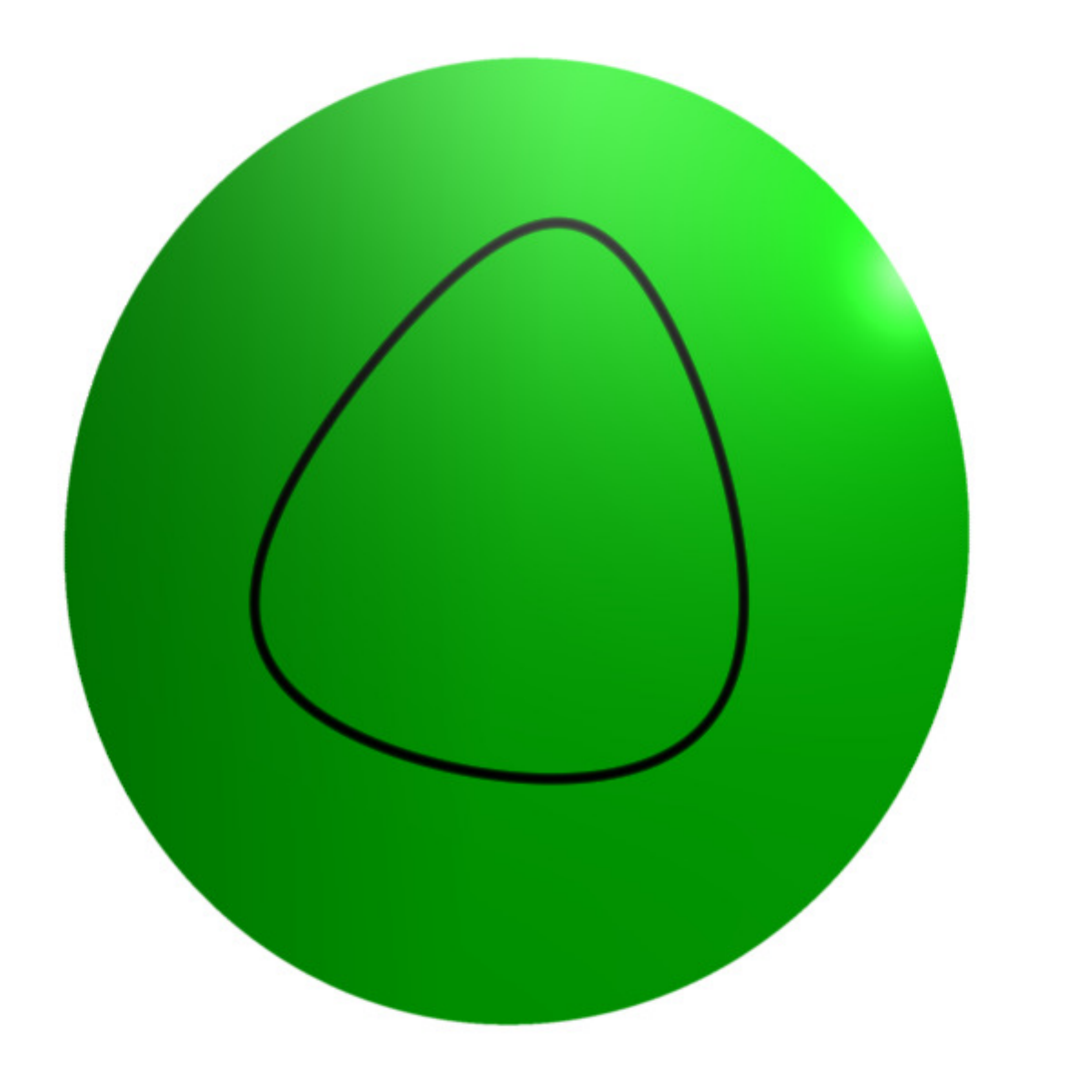}
 \caption{The convex body $C$ of canonical parameters
 is intersected  with a linear space $\mathcal{L}$.
 Dual to the intersection  $C \cap \mathcal{L}$ is
 the projection of  the exponential variety   shown in Figure~\ref{fig:redblue}.  }
 \label{fig:yellowgreen}
\end{figure}

We now finally come to the definition of exponential varieties.
They arise when restricting a hyperbolic exponential family to a  subspace $\mathcal{L}$.
They are analogous to the toric varieties above.

\begin{definition}[({\em Exponential variety})]\label{df:EV} \rm
Let $f(\theta)$ be a hyperbolic polynomial,
with hyperbolicity cone $C\subset \mathbb{R}^d$ and dual cone $K = C^\vee$. The corresponding \emph{exponential variety}, denoted 
$\, \mathcal L^{F} = \overline{F(C\cap \mathcal{L})}$,
is the Zariski closure of the image of $\mathcal{L}$ under the gradient map
$F = -\nabla \log(f)$.
\end{definition}

Here the Zariski closure could be taken in $\R^d$, but
we usually prefer to work in $\C \PP^{d-1}$ instead. In this case, like in (\ref{eq:ratlmap2}), we can equivalently study the image of $\nabla f$.
If $c = {\rm dim}(\mathcal{L})$ then $ \mathcal L^{F}$ will be
a $(c-1)$-dimensional projective variety in $\C \PP^{d-1}$.
The {\em positive exponential variety}
is the semialgebraic set $\mathcal L_{\succ 0}^{F}=F(C\cap \mathcal L)$.
This is analogous to the toric model. It lives in
the space $K \subset \R^d$ of sufficient statistics, or preferably,
in its image in $\R \PP^{d-1}$.

A key property of the exponential variety $\mathcal{L}^F$ is that the canonical bijection between
$ C_\mathcal{L}$ and ${\rm int}(K_\mathcal{L})$
   factors through the positive part $\mathcal L_{\succ 0}^{F}$.
This is the content of the following result.

\begin{theorem} \label{thm:sec4main}
Fix a hyperbolic exponential family, with $f , C,K$ and $ F = -\nabla \log(f)$ as above, and 
let $\mathcal L$ be a linear subspace of $\R^d$ that intersects
 $C$. Then the gradient map $F_{\mathcal L}$ of the restricted
  exponential family on $\mathcal L$ can be written as a sequence of maps
\begin{equation}
\label{eq_maps}
C_{\mathcal L}\;\;\subset\;\; C\;\;\overset{F}{\longrightarrow} \;\;K\;\;\overset{L}{\longrightarrow} \;\;K_{\mathcal L}.
\end{equation}
The convex set $C_\mathcal{L}$
of canonical parameters maps bijectively to
the positive exponential variety $\mathcal L_{\succ 0}^{F}$,
and $\mathcal L_{\succ 0}^{F}$ maps bijectively to the interior of the convex set $K_\mathcal{L}$
of sufficient statistics.
\end{theorem}

\begin{proof} Let $\tau$ be the coordinates on $C_{\mathcal L}$ and  $\theta$ the coordinates on $C$. The projection $\pi_\mathcal{L}: K\to K_{\mathcal L}$ is defined by
the $c \times d$-matrix $L$, and the embedding $C_{\mathcal L}\subset C$ is given by $\theta=L^{T}\tau$. From (\ref{eq:FLtheta}) we know that $L\circ F\circ L^{T}=F_{\mathcal L}$.
This is precisely what is being asserted in (\ref{eq_maps}).
 By Theorem~\ref{thm:brown} and Remark~\ref{rem_Riesz}, the two gradient
 maps $F$ and $F_{\mathcal L}$ are both bijections. Since $C_{\mathcal L}\subset C$ is an embedding, the projection $\pi_{\mathcal L}$ must be a bijection when restricted to $F(C_{\mathcal L})
 = \mathcal L_{\succ 0}^{F}$.
 \end{proof}

 \begin{corollary} \label{eq:samedimension}
 The dimension of the exponential variety $\mathcal{L}^F$ equals
 that of the space~$\mathcal{L}$.
 \end{corollary}

\begin{remark} In defining exponential varieties we chose
to restrict to hyperbolic polynomials.
A more inclusive definition would start with an arbitrary 
rational exponential family $(\cX,\nu,T)$ with
  gradient map $F:C\to K$. For any subspace $\mathcal{L}\subset\mathbb{R}^d$
 with $C \cap \mathcal{L}\neq\emptyset$, one might then refer to
 $\, \mathcal L^{F} = \overline{F(C\cap \mathcal{L})}\,$ as an exponential variety.
 Theorem \ref{thm:sec4main} remains true in that setting.
But, even that definition does not yet include toric varieties.
Indeed, the full discrete exponential family is not rational, and
that is why $L$ was an integer matrix
in Example \ref{ex:rationaldiscrete1}.
It could make sense to work with a 
notion of exponential variety that includes toric varieties,
for instance by working in the setting
of algebraic exponential families as
in \cite{DS}. However, we chose not to pursue that path in
the present paper.
We decided to use the term `exponential variety' exclusively
in the sense of Definition \ref{df:EV}.
Our focus is entirely on hyperbolic polynomials $f(\theta)$.
\end{remark}

Given a hyperbolic polynomial $f$, its restriction 
 $f|_\mathcal{L}$ is a hyperbolic polynomial
 on the linear subspace $\mathcal{L} \subset \R^d$.
The hyperbolicity cone of $f_\mathcal{L}$ equals
$\, C_\mathcal{L} = C \cap \mathcal{L}$, and its dual is the cone
\begin{equation}
\label{eq:fivecones}
 K_\mathcal{L} \,\,=\,\,
(C_\mathcal{L})^\vee \,\,= \,\,
 (C \cap \mathcal{L})^\vee \,\,= \,\,\pi_\mathcal{L} (C^\vee) \,\, = \,\, \pi_\mathcal{L}(K).
 \end{equation}

From the perspective of algebraic geometry, we could also define
the exponential variety
$\mathcal{L}^{\nabla f}$ for any
homogeneous $f \in \R[\theta_1,\ldots,\theta_d]$ and any
subspace $\mathcal{L} \subset \C \PP^{d-1}$.
Namely,  $\mathcal{L}^{\nabla f}$ is
the closure of the image of $\mathcal{L}$ under
the gradient map in (\ref{eq:ratlmap2}).
This level of generality makes perfect sense
for the study of algebraic degrees as in Section \ref{sec:ml_degree}.
However, in order for $\mathcal{L}^{\nabla f}$ to have
a distinguished positive part $\mathcal{L}^{\nabla f}_{\succ 0}$,
with its remarkable role as the ``middleman'' in the bijection
$F_\mathcal{L}: C_\mathcal{L}  \rightarrow {\rm int}(K_\mathcal{L})$
between two dual convex sets, we need that $f$ is hyperbolic.

We note that  Conjecture \ref{conj_Riesz} 
remains valid after restricting to a linear section.

\begin{proposition}
\label{prop:restrictBHWC}
Let $f$ be a hyperbolic polynomial with hyperbolicity cone $C$. Suppose that the Riesz kernel 
$q_\alpha(\sigma)$ is nonnegative for $f$ and some $\alpha>0$. Fix a linear subspace $\mathcal{L}$ that intersects the interior of $C$. Then the Riesz kernel for $f|_{L\cap C}$ is also nonnegative for $\alpha$.
\end{proposition}
\begin{proof}
By the Bernstein-Hausdorff-Widder-Choquet Theorem we know that $f^{-\alpha}$ is completely monotone on $C$. The inequalities (\ref{def:cm}) remain true after restricting to the subcone 
$\mathcal{L}\cap C$, hence $f^{-\alpha}|_{ \mathcal{L}\cap C}$ is completely monotone. Again, by Bernstein-Hausdorff-Widder-Choquet we know that there exists a measure $\mu$ on 
the dual cone $\,K_\mathcal{L} = ( C \cap \,\mathcal{L})^\vee\,$ such that
$$ f(\theta)^{-\alpha}\,\,=\,\,
\int_{(\mathcal{L}\cap C)^\vee}\exp(-\langle \theta,\sigma\rangle)\;\mu(\dd \sigma).$$
By the uniqueness of the inverse Laplace transform, the measure $\mu$ must be induced by the Riesz kernel $q_\alpha(\sigma)$ as in (\ref{eq:intrep2}).
This means that $q_\alpha(\sigma)$ cannot take negative values. 
\end{proof}

The most prominent  hyperbolic polynomial is the determinant of a symmetric matrix
of linear forms. Conjecture \ref{conj_Riesz}  holds for such Gaussian models,
by Propositions \ref{prop_Wishart} and \ref{prop:restrictBHWC}.

\begin{example}[({\em Linear Gaussian concentration models})]  \rm
\label{ex:lineargaussian}
Consider the multivariate Gaussian family described in Example~\ref{ex:fullgaussian},
with $C = {\rm PD}_m$ and $K = \overline{{\rm PD}}_m$ in the space
$\S^m$ of real symmetric $m \times m$-matrices.
 A linear subspace $\mathcal{L}$ with $\mathcal{L}\cap C
  \neq\emptyset$ defines a {\em linear Gaussian concentration model}. The corresponding exponential variety
  $\mathcal{L}^{\nabla f}$ was studied in~\cite{stuhl},
  where it was denoted by $\mathcal{L}^{-1}$.
   Its positive part consists of all covariance matrices in the model, i.e.,
$$\mathcal L^{-1}_{\succ 0} \,\,= \,\,\{\sigma \in {\rm PD}_m \;\mid\; \sigma^{-1} \in\mathcal{L}\}. $$
The instances of most interest are the {\em Gaussian graphical models} in \cite{Dempster,Lauritzen} where $\mathcal{L}$ is defined by the vanishing of some of the off-diagonal entries $\theta_{ij}$. 
A natural extension is the class of
colored Gaussian graphical models  in
 \cite{HLcolor}. The case when $\mathcal{L}$ consists of diagonal matrices is studied in
\cite[Section 3]{stuhl}.
For a concrete example, let $m=4$ and consider the subspace
$$
\mathcal{L} \,\, = \,\, \{\,
{\rm diag}( \tau_{1} , \tau_{2} , \tau_{1}-\tau_{2} , \tau_{1}+\tau_{2})
\,:\,\,
\tau \in \R^2 \,\}.
$$
The exponential variety $\mathcal{L}^{-1}$ is a cubic curve in the $\C \PP^3$
of diagonal matrices $\sigma$, namely
$$ \mathcal{L}^{-1} \, = \,
\bigl\{ \sigma \in \C \PP^3: \,
\sigma_1 \sigma_2 - \sigma_1 \sigma_4 - \sigma_2 \sigma_4 =
 \sigma_1 \sigma_3 + \sigma_1 \sigma_4 - 2 \sigma_3 \sigma_4 =
\sigma_2 \sigma_3- \sigma_2 \sigma_4 - 2 \sigma_3 \sigma_4  = 0 \bigr\}.
$$
The positive part $\mathcal{L}^{-1}_{\succ 0}$
consists of $\sigma \in \mathcal{L}$ where all  $\sigma_i$ are real and positive.
Writing  $L$ for the linear map given by
 $\rho_1 =  \sigma_1 + \sigma_3 + \sigma_4$ and
$\rho_2 =  \sigma_2 - \sigma_3 + \sigma_4$,
the bijections in (\ref{eq_maps}) are
 $$
 C_\mathcal{L}=\{\tau \in \R^2: \tau_{1}>\tau_{2}>0\}  \,\,
  \overset{\nabla f}{\longrightarrow}\,\,
  \mathcal{L}^{-1}_{\succ 0} \,\,
    \overset{L}{\longrightarrow} \,\,K_{\mathcal L}
    =\{ \rho \in \R^2: \rho_{1}\geq 0,\rho_{1}+\rho_{2}\geq 0\}.
    $$
The Riesz kernel for this particular exponential variety was
computed in Example  \ref{ex_diag_Gaussian}.
 \hfill $\diamondsuit$
 \end{example}

\begin{figure}[h]
\centering
  \vspace{-0.5in}
 \includegraphics[width=5.8cm]{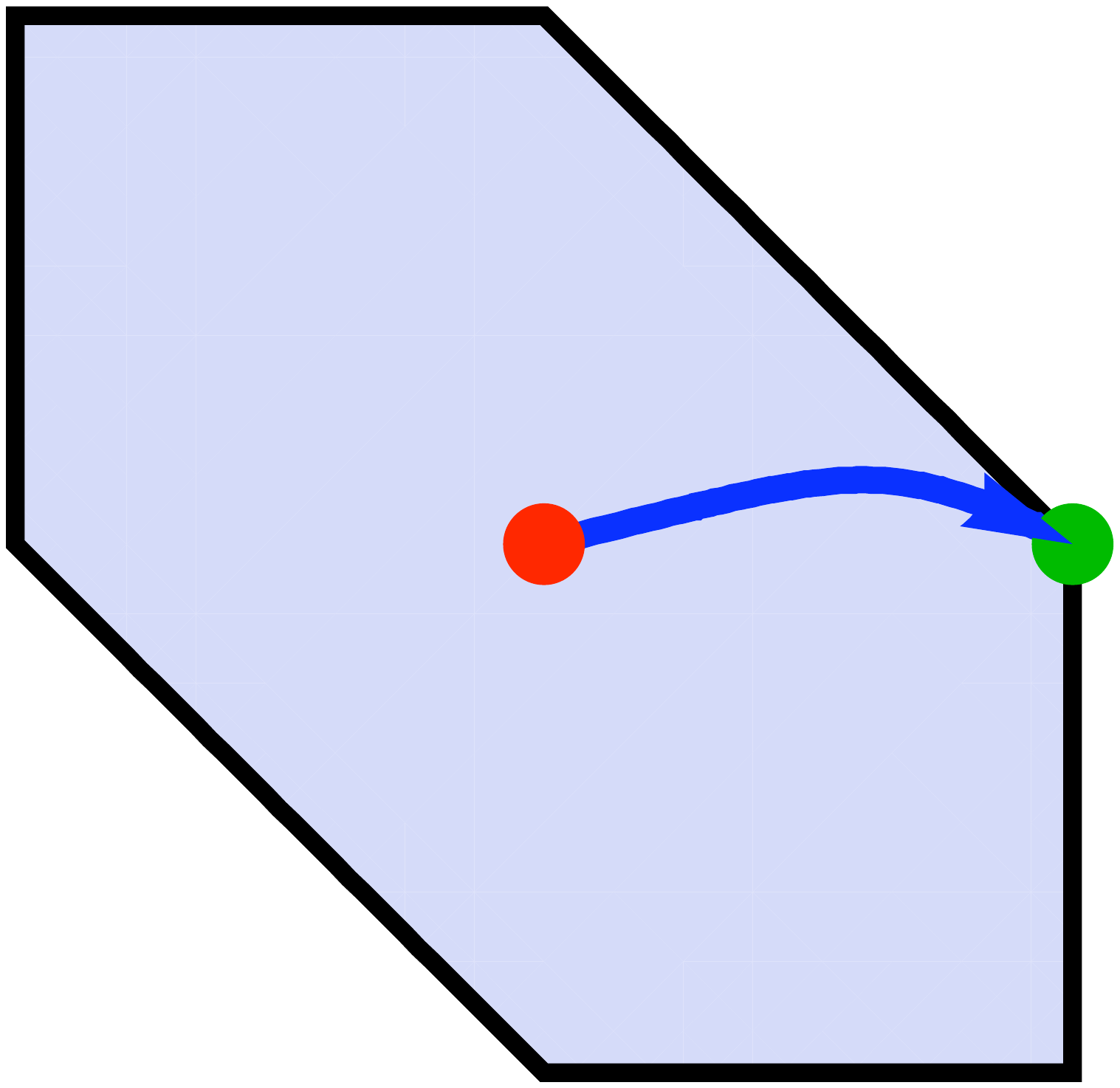}   \qquad
  \includegraphics[width=6.1cm]{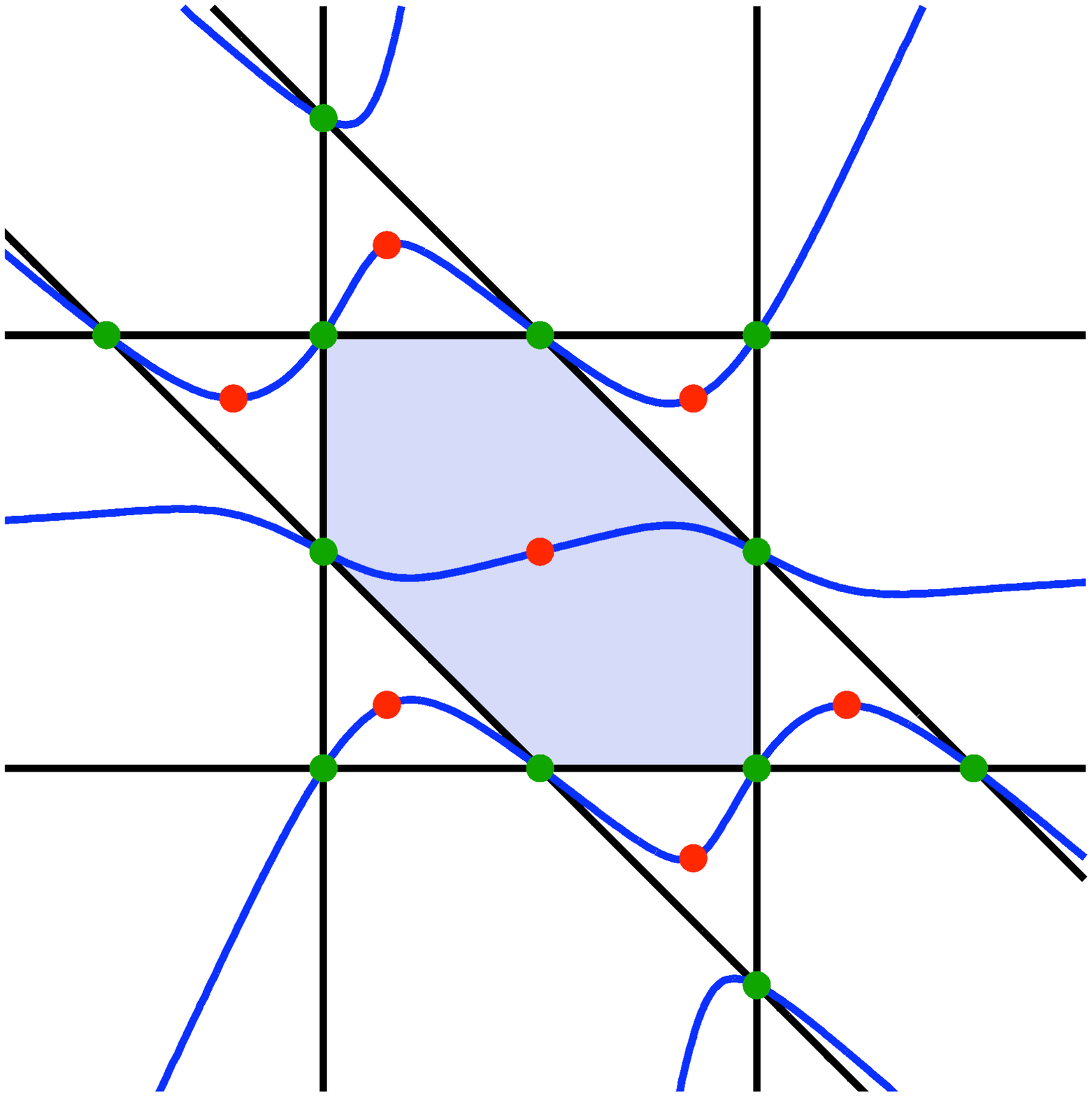}
  \vspace{-0.5in}
 \caption{Central curves in linear and semidefinite programming
 are exponential varieties.}
  \label{fig:centralcurve}
\end{figure}

\begin{example}[({\em The central path is an exponential curve})] \rm
\label{ex:optimization}
Let $f$ be a hyperbolic polynomial and
$C$ its hyperbolicity body in $\R \PP^{d-1}$.
 If $c=2$ then
$\mathcal{L}^{\nabla f}_{\succ 0}$ is
a curve inside the dual body~$K$.
This curve represents a {\em central path}
(cf.~\cite{guler2}) for
linear optimization over $K$.
The case of linear programming
was studied in \cite{DSV}. Here
$K$ is a convex polytope and
$K_\mathcal{L}$ the line segment of
possible values of the  objective function.
The central path and the exponential curve
are shown in Figure~\ref{fig:centralcurve} for the case when
 $d=3$, $c=2$, and $f(\theta)$ is a
product of six linear forms.~$\diamondsuit$
\end{example}

 We discuss two further examples
 of combinatorial interest. These are based on \cite{BrandenES,kummer,sokal}.

\begin{example}[({\em Graphs})]
\label{ex:graphs}
\rm
We fix a connected simple graph $G$ with vertex set
$ \{0,1,\ldots,m \}$ and edge set $E(G)$. The {\em reduced Laplacian} of $G$ is
a symmetric $m \times m$ matrix $\Lambda_G $ with
 rows and columns indexed by $\{1,\ldots,m\}$.
Its  nonzero off-diagonal  entries are $-\theta_{ij}$, and
 the $i$-th row sum of $\Lambda_G$ equals $\theta_{0i}$.
 By convention,  $\theta_{ij}=0$ if $\{i,j\}\notin E(G)$.
 The {\em spanning tree polynomial} $\,f(\theta) = {\rm det}(\Lambda_G)\,$
   is  a complete hyperbolic polynomial. We fix the exponential family for $f$.

 Now, let $H$ be any subgraph of $G$ that is also connected.
 We consider the linear subspace
$$\mathcal{L} \,= \,\bigl\{ \theta \in \R^{E(G)} \,:\,\,\theta_{ij} = 0 \,
\,\, \,\hbox{for}\,\, \{i,j\} \in E(G) \backslash E(H) \,\bigr\}.$$
The exponential variety $\mathcal{L}^{\nabla f}$ is an invariant
of the pair $H \subset G$.
It lives in $\C \PP^{|E(G)|-1}$ and its dimension is $|E(H)|-1$.
The players of Theorem  \ref{thm:sec4main}
are quite interesting in this~case.

For a concrete example, let $m=3$ and  $G$  the
complete graph $K_4$, with reduced Laplacian
\begin{equation}
\label{eq:K4laplace}
 \Lambda_{K_4} \,\,= \,\,
\begin{bmatrix}
 \theta_{01}+\theta_{12}+\theta_{13} & -\theta_{12} &        -\theta_{13} \\
  -\theta_{12} &        \theta_{02} +\theta_{12} +\theta_{23} &  -\theta_{23} \\
  -\theta_{13} &       -\theta_{23} &       \theta_{03} +\theta_{13} + \theta_{23}
\end{bmatrix}.
\end{equation}
The cubic polynomial $f(\theta)  = {\rm det}(\Lambda_{K_4})$ has
$16$ monomials, one for each spanning tree of $K_4$.
Every connected subgraph $H$ of $K_4$ defines an
exponential variety $\mathcal{L}^{\nabla f}$ in $\C \PP^5$.
For instance, let $H$ be the $4$-cycle, with
$\mathcal{L} = \{\theta_{02}=\theta_{13} =0\}$.
Then $ \mathcal{L}^{\nabla f} $ is a threefold of degree $4$,
obtained as the complete intersection of two quadrics,
that is singular along three lines.
On the other hand, if $H$ is the $4$-chain with
$\mathcal{L} =  \{\theta_{02}=\theta_{13} =\theta_{03}=0\}$,
then the exponential variety is a plane, namely
$ \mathcal{L}^{\nabla f}  = \{ \sigma \in \C \PP^5 :
\sigma_{12}-\sigma_{13}+\sigma_{23} =
\sigma_{02}-\sigma_{03}+\sigma_{23} =
\sigma_{01} - \sigma_{03} + \sigma_{13} = 0 \} $.
 \hfill $\diamondsuit$
\end{example}

\begin{example}[({\em V\'amos matroid})]
\label{ex:vamos} \rm
Let $d = 8$ and
$f(\theta) = \sum \theta_i \theta_j \theta_k \theta_l $
where the sum runs over all quadruples
in $\{1,2,3,4,5,6,7,8\}$ other than
$1256$, $1278$, $3456$, $3478$, $5678$.
So, $f(\theta)$ has $65$ terms.
This is the basis generating function
of the {\em V\'amos matroid}, a popular hyperbolic polynomial.
The specialization studied by Kummer \cite{kummer}
 is the restriction of $f$ to the subspace
$\mathcal{L} = \{\theta \in \R^8:
\theta_1 = \theta_2, \,\theta_3 = \theta_4, \,\theta_5 = \theta_6, \,\theta_7 = \theta_8\}$.
The main result in \cite{kummer} states that
the $3$-dimensional  body $C_\mathcal{L}$ is a spectrahedron.
The bijection from $C_\mathcal{L}$ to the interior  of
$K_\mathcal{L} = C_\mathcal{L}^\vee$ factors through
the positive variety $\mathcal{L}^{\nabla f}_{\succ 0} \subset \R \PP^7$.
On this particular plane $\mathcal{L}$, the gradient map $\nabla f$ is an involution, that is,
$\,\mathcal{L}^{\nabla f}
= \{\sigma_1 = \sigma_2, \,\sigma_3 = \sigma_4, \,\sigma_5 = \sigma_6, \,\sigma_7 = \sigma_8\}$.
However, for generic choices of $\mathcal{L}$,
the exponential variety $\mathcal{L}^{\nabla f}$ is
a non-linear threefold in $\C \PP^7$. For instance, if
we replace $\theta_8$ by $-\theta_8$ in $\mathcal{L}$ then
$\mathcal{L}^{\nabla f}$ is a hypersurface of degree $18$ in
a $4$-plane in $\C \PP^7$. If we further replace  $\theta_6$ by $-\theta_6$    then
$\mathcal{L}^{\nabla f}$ has codimension $2$ and degree $10$
in a $5$-plane in $\C \PP^7$.
 \hfill $\diamondsuit$
\end{example}

In Examples \ref{ex:graphs} and \ref{ex:vamos} we saw
that the exponential variety $\mathcal{L}^{\nabla f}$
can be linear again when $\mathcal{L}$ is chosen in a very special manner.
This raises the question
for which pairs $(f,\mathcal{L})$ this happens.
The remainder of this section is devoted to
what we know about this question.

The case of linear Gaussian concentration models, where $f(\theta)=\det(\theta)$,
 $C = {\rm  PD}_m$ and $\mathcal{L}^{\nabla f} = \mathcal{L}^{-1}$,
  was studied by Jensen in \cite{Jensen}.
He proved that the exponential variety $\mathcal{L}^{-1}$ is linear if and only if $\mathcal{L}$ is a
{\em Jordan algebra}, i.e., $\theta^2\in \mathcal{L}$ for all $m\times m$ matrices $\theta$ in
$\mathcal{L}$. Moreover, Jensen showed that if $\mathcal{L}^{-1}$ is linear
then $\mathcal{L}^{-1}=\mathcal{L}$
under the identification of $\sigma$ with $\theta$.

The situation is not as nice and clean for  hyperbolic polynomials $f(\theta)$ other than the
symmetric determinant. As we have seen in Example~\ref{ex:graphs}, 
 a linear exponential variety $\mathcal{L}^{\nabla f}$
can be quite different from its underlying
 $\mathcal{L}$. It would be very desirable to
find a characterization that generalizes  Jensen's results on Gaussian families to
  hyperbolic exponential families.

As a first start, we examine the special case of codimension $1$.
The following result explicates the
conditions under which an exponential variety $\mathcal{L}^{\nabla f}$
happens to be a hyperplane.

\begin{proposition}
Fix a hyperbolic polynomial $f$ and suppose that the hyperplane
$\bigl\{\sum_{i=1}^d a_i\sigma_i=0 \big\}$ is an exponential variety  for $f$.  Then the polynomial $\sum_{i=1}^d a_i (\partial f/\partial\theta_i)$ is divisible by a linear form $\ell(\theta)$. Moreover, if  the
hyperplane $\mathcal{L} = \{\ell(\theta)=0\}$ satisfies the conclusion of Corollary~\ref{eq:samedimension} --
for instance because $\mathcal{L} \cap C \not = \emptyset$ --
then  $\mathcal{L}^{\nabla f} =  \bigl\{\sum_{i=1}^d a_i\sigma_i=0\bigr\}$.
\end{proposition}

\begin{proof}
Let $\mathcal{H}$ be
 the hyperplane given by $\sum_{i=1}^d a_i\sigma_i=0$. Suppose that $\mathcal{H}$ is an exponential variety $\mathcal{L}^{\nabla f}$, where $\mathcal{L}$ is yet to be specified.
The polynomial $\sum_{i=1}^d a_i \frac{\partial f}{\partial\theta_i}$ defines
a hypersurface polar to the hyperbolic polynomial $f$. A point $\theta \in \C \PP^{d-1}$ lies on that
polar hypersurface if and only if $F(\theta) \in \mathcal{H}$.
Since the hyperplane $\mathcal{L}$ maps to $\mathcal{H}$,
it is a subset of the polar hypersurface. This means that
$\sum_{i=1}^d a_i \frac{\partial f}{\partial\theta_i}$ is divisible
by the linear equation that defines $\mathcal{L}$.

For the second statement, note that if the linear form $\ell(\theta)$ divides $\sum_{i=1}^d a_i \frac{\partial f}{\partial\theta_i}$, then the hyperplane $\mathcal{L}$ defined by $\ell(\theta) $ maps to $\mathcal{H}$. Assuming that the image of $\mathcal{L}$ has codimension one in $(\C \PP^{d-1},\sigma)$, as expected,
then the closure of that image must coincide with $\mathcal{H}$.
\end{proof}

We now show how these conditions can be checked
for the cubic in our running example.

\begin{example}\rm
Consider the elementary symmetric polynomial $E_3$ in four variables, as in
Examples \ref{ex:E_3}, \ref{ex:E_3f},
\ref{ex:E_3b}, \ref{ex:E_3c}. We need to analyze the
quadrics that are polar to the cubic $E_3$:
$$
\sum_{i=1}^4 a_i\frac{\partial E_3}{\partial\theta_i} \quad = \quad \frac{1}{2}
\begin{pmatrix} \theta_1 \! & \! \theta_2\! & \! \theta_3\! & \!\theta_4 \end{pmatrix}
\begin{pmatrix}
0&a_3+a_4&a_2+a_4&a_2+a_3\\
a_3+a_4&0&a_1+a_4&a_1+a_3\\
a_2+a_4&a_1+a_4&0&a_1+a_2\\
a_2+a_3&a_1+a_3&a_1+a_2&0\\
\end{pmatrix}
\begin{pmatrix} \theta_1 \\ \theta_2 \\ \theta_3 \\ \theta_4 \end{pmatrix}
$$
The quadric has a linear factor if and only if
all $3\times 3$ minors of this symmetric matrix vanish. Up to scaling and permuting coordinates,
 this happens only for two points, namely
$a = (-1:1:1:1)$ and $a =(1:-1:0:0)$. In the first case, the quadric is $\theta_1(\theta_2+\theta_3+\theta_4)$, providing us with two planes $\mathcal{L}$ that yield the same exponential variety
$\mathcal{L}^{\nabla E_3} = \{\sigma_1 = \sigma_2+\sigma_3+\sigma_4\}$.
 In the second case the quadric is $(\theta_2-\theta_1)(\theta_3+\theta_4)$. Each factor yields
 $\mathcal{L}^{\nabla E_3} = \{\sigma_1 = \sigma_2\}$.

   It is straightforward to generalize this construction to other polynomials of degree $3$.
\hfill $\diamondsuit$
\end{example}

The situation is more difficult when the given polynomial $f$ has
 degree $m \geq 4$. To find all hyperplanes that are exponential for $f$,
we must eliminate the quantifiers in the formula
$$\exists \, \ell \,\,\hbox{linear}\,:\, \ell(\theta)\,\,\, \hbox{divides} \,\,\,
a_1 \frac{\partial f}{\partial\theta_1}  + \cdots +
a_d \frac{\partial f}{\partial\theta_d}  .$$
To this end, we consider the variety of all homogeneous polynomials
of degree $m-1$ in $d$ unknowns that have a linear factor.
This variety is denoted ${\rm Chow}_{1,m-2}$. It is a variant of the
 {\em Chow variety} of polynomials that are products of linear forms. This is a classical topic
in algebraic geometry with recent connections  to computational complexity \cite[Section~6.4]{BLMW}.

For a given $f$, we must intersect the linear space spanned by
the partial derivatives of $f$ with ${\rm Chow}_{1,m-2}$.
Each intersection point furnishes a candidate for an exponential hyperplane.
For instance, let $d=3$ and $m=4$, so $f$ is a ternary quartic.
Then $\frac{\partial f}{\partial \theta_1},
\frac{\partial f}{\partial \theta_2},\frac{\partial f}{\partial \theta_3}$
span a $2$-plane in the $\C \PP^9$ of plane cubics.
We must intersect that plane with ${\rm Chow}_{1,2}$, which is
a variety of codimension $2$ and degree $21$ in $\C \PP^9$. Its prime ideal
is generated by $35$ polynomials of degree $8$.
The $21$ points $(a_1:a_2:a_3)$ in that intersection represent
all lines $\{a_1 \sigma_1 + a_2 \sigma_2 + a_3 \sigma_3 = 0\}$
in $\C \PP^2$ that can have the structure of an exponential variety with respect to
the quartic $f$.

\section{Gradient multidegree and ML degree}
\label{sec:ml_degree}

In this section we introduce algebraic complexity measures
for hyperbolic exponential families.
The partition function  is  $f^{-\alpha}$
where $\alpha > 0$ and $f$ is a hyperbolic polynomial.
We are interested in the rational map  $F : \C \PP^{d-1} \dashrightarrow \C \PP^{d-1}$
defined by the gradient of the log-partition function $A(\theta) = - \alpha  \cdot {\rm log} \,f(\theta)$.
We shall work in the
convenient setting of projective geometry. This means that
the logarithm and the constant $\alpha$ are ignored.
From now on, $F $ is simply the rational  map (\ref{eq:ratlmap2}) defined by the gradient
   of the polynomial $f(\theta)$.

   The following definition makes sense for any homogeneous polynomial $f(\theta)$.
   However, in our context we always assume that $f(\theta)$
is hyperbolic, that the  hyperbolicity cone $C$ is pointed in $(\R^d,\theta)$
and that its dual $K$ is a pointed cone in $(\R^d, \sigma)$.
We regard $C$ and $K$ as convex bodies in
(affine charts of) the real projective spaces
$(\R \PP^{d-1}, \theta)$ and $(\R \PP^{d-1}, \sigma)$.
The gradient map $F  =\nabla f$ defines a bijection between the
interiors of these two convex bodies. Its inverse is an algebraic function, and we
are interested in the degree of that function.

\begin{definition} \rm
Let $f$ be a homogeneous polynomial
in $\theta_1,\ldots,\theta_d$, and let $X_f$ be the graph of its
gradient map $F = \nabla f$. Thus $X_f$ is the closure in
 $\C \PP^{d-1} \times \C \PP^{d-1}$ of the set of points $(\theta, \sigma)$, where
$\theta $ lies in $\C \PP^{d-1}$ but is not
in the singular locus of  the hypersurface $\{f=0\}$, and
$\sigma = F(\theta)$. By construction, the graph $X_f$ is an irreducible
subvariety of dimension $d-1$ in  $\C \PP^{d-1} \times \C \PP^{d-1}$.
We define the {\em gradient multidegree} of the polynomial $f$ to be
the class $[X_f]$ of that variety in the integral cohomology ring
of the product of the two projective spaces:
 $$ H^* \bigl(\C \PP^{d-1} \times \C \PP^{d-1}; \mathbb{Z} \bigr) \,\,  = \,\,
\mathbb{Z}[\,s,\,t\,]/\langle s^d, t^d \rangle . $$
The ring generators are the divisor classes
$s =  [\C \PP^{d-1} \times \{\sigma_i = 0\}]$ and
$t  = [\{\theta_j = 0\} \times \C \PP^{d-1}]$.

In concrete geometric terms, the gradient multidegree of $f$ is the generating function
$$ [X_f] \,\,\, = \,\,\,
\alpha_d s^{d-1}  + \alpha_{d-1} s^{d-2} t +
\alpha_{d-2} s^{d-3} t^2 + \cdots + \alpha_2 s t^{d-2} + \alpha_1 t^{d-1}, $$
where $\alpha_i$ is the cardinality of the finite set
$X_f \cap (L_{i-1} \times M_{d-i} )$ where
$L_{i-1}$ and $M_{d-i}$ are general subspaces
in $\C \PP^{d-1}$
of dimensions $i-1$ and $d-i$ respectively.
The leading coefficient $\alpha_d$ of  $[X_f]$ is
the degree of the map (\ref{eq:ratlmap2}).
We call this the {\em gradient degree} of $f$.
\end{definition}

In statistical applications, the quantity $\alpha_d $
measures the algebraic complexity of the function $\sigma \mapsto \theta$
 that take the sufficient statistics to the natural parameters. It
 plays the role of the ML degree for the exponential family in question.
 However, in this paper we reserve the term ``ML degree''
 for exponential varieties, as in Definition \ref{def:MLdegree} below.
 This is consistent with the usage in \cite{stuhl}.
The  number $\alpha_d = {\rm degree}(F)$ is
the {\em gradient degree of the hyperbolic exponential family}.

\begin{example}
\label{ex:E_3b}
\rm
Let $d=4$ and fix the elementary symmetric polynomial as in Example \ref{ex:E_3},
$$ f \,\, = \,\,  \theta_1 \theta_2 \theta_3
+ \theta_1 \theta_2 \theta_4
+ \theta_1 \theta_3 \theta_4
+ \theta_2 \theta_3 \theta_4 .$$
The gradient multidegree of $f$ is the following class in the integral cohomology of $\C \PP^3 \times \C \PP^3$:
\begin{equation}
\label{eq:4421}
 [X_f] \,\, = \,\,
4 s^3 \,+\, 4 s^2 t \,+\, 2 s t^2 \,+\,1 t^3.
\end{equation}
We see that the gradient degree of the polynomial $f$ is $4$. Hence we can use
Cardano's formula to express the
 MLE $\hat \theta$  in terms of radicals in the sufficient statistics
$\sigma$ of the data.

To be explicit, the graph $X_f$ is the subvariety of $\C \PP^3 \times \C \PP^3$
defined by the $2 \times 2$-minors~of
\begin{equation}
\label{eq:twobyfour}
\begin{pmatrix}
 \sigma_1 & \sigma_2 & \sigma_3 & \sigma_4 \\
\theta_2 \theta_3 + \theta_2 \theta_4 + \theta_3 \theta_4 &
\theta_1 \theta_3 + \theta_1 \theta_4 + \theta_3 \theta_4 &
\theta_1 \theta_2 + \theta_1 \theta_4 + \theta_2 \theta_4 &
\theta_1 \theta_2 + \theta_1 \theta_3 + \theta_2 \theta_3
\end{pmatrix},
\end{equation}
over the nonsingular locus of $\{f=0\}$, and taking the Zariski closure of the resulting set.

If we replace the unknowns $\sigma_1,\sigma_2,\sigma_3,\sigma_4$ by
fixed generic real numbers, then we get a system of equations in
$\theta = (\theta_1: \theta_2:\theta_3:\theta_4)$ that has
four complex solutions.
Assuming that $\sigma$ is in
the interior of the red convex body $K$, precisely one
of these solutions lies in the yellow convex body $C$.
The colors refer to Figure \ref{fig:somosa}.
This unique solution is the MLE $\,\hat \theta$.
\hfill $\diamondsuit$
\end{example}

\begin{remark} \rm
If $f = {\rm det}(\theta)$ is the symmetric
$m \times m$-determinant, representing the full Gaussian family
 in Example \ref{ex:fullgaussian}, then $[X_f]$ is a binary form 
  that is known in some cases. 
By \cite[Theorem 2.3]{stuhl}, it is symmetric under swapping $s$ and $t$.
 For instance, for $m=4$,
 $$ [X_f] \,= \,1 s^9 + 3 s^8 t + 9 s^7 t^2 + 17 s^6 t^3 + 21 s^5 t^4
 + 21 s^4 t^5 + 17 s^3 t^6 + 9 s^2 t^7 + 3 s t^8 + 1 t^9. $$
 At present, no general formula is known
for the gradient multidegree of the symmetric 
$m \times m$-determinant. Expressions
 for some of its coefficients are derived in \cite{stuhl} after Theorem 2.3.
\end{remark}

In the computer algebra system  {\tt Macaulay2}  \cite{M2},
we can compute $[X_f]$ from a given hyperbolic polynomial $f$
using the built-in command {\tt multidegree}.
See \cite[Section 8.5]{CCA} for an explanation of how multidegrees
in commutative algebra
represent cohomology classes on toric varieties.
Here is a piece of {\tt Macaulay2} code that generates
(\ref{eq:twobyfour}) and the output (\ref{eq:4421}):
\smallskip
\begin{verbatim}
R = QQ[t1,t2,t3,t4,s1,s2,s3,s4,
       Degrees => {{1,0},{1,0},{1,0},{1,0},{0,1},{0,1},{0,1},{0,1}}];
f = t1*t2*t3 + t1*t2*t4 + t1*t3*t4 + t2*t3*t4 ;
gradf = {diff(t1,f),diff(t2,f),diff(t3,f),diff(t4,f)};
M = matrix { {s1,s2,s3,s4}, gradf };
multidegree saturate( minors(2,M), ideal(gradf) )
\end{verbatim}
\smallskip

This method works only for small examples
because the command {\tt saturate} is slow.
To compute the gradient multidegree for larger polynomials $f(\theta)$,
it is better to use the formula
\begin{equation}
\label{eq_points}
\alpha_i \, \,\, = \,\,\,\# \bigl(X_f \cap (L_{i-1} \times M_{d-i})  \bigr).
\end{equation}
To count the number of points in this intersection,
one can either use Gr\"obner bases over a finite field,
or the homotopy methods supplied by
numerical algebraic geometry \cite{bertini}.

We now turn our discussion to exponential varieties.
Let $f$ be a fixed hyperbolic polynomial with hyperbolicity cone $C \subset \R^d$.
We consider any linear subspace  $\mathcal{L}$ which intersects the
interior of $C$, and we set $c = {\rm dim}(\mathcal{L})$.
The corresponding exponential variety in $\C \PP^{d-1}$ is
 $\,\mathcal{L}^{\nabla f} := \overline{F(\mathcal{L})}$. It follows from the results
 in the previous section that $\,{\rm dim}(\mathcal{L}^{\nabla f}) = c-1$ and the positive exponential variety $\, \mathcal{L}^{\nabla f}_{\succ 0} :=
\mathcal{L}^{\nabla f} \cap {\rm int}(K)$ maps bijectively onto the interior of
 $\,(C \cap \mathcal{L})^\vee =\pi_\mathcal{L}(K)\,$
under the linear map $\pi_\mathcal{L}$. The inverse of that map
is an algebraic function. We are interested in the degree.
This can alternatively be described as follows.

\begin{definition}\label{def:MLdegree} \rm
Let $\,\pi_\mathcal{L} : \C \PP^{d-1} \dashrightarrow  \C \PP^{c-1}$
denote the projection with center $\mathcal{L}^\perp$.
This restricts to a finite-to-one map
from the exponential variety $\mathcal{L}^{\nabla f}$ onto
$ \C \PP^{c-1}$. We define
$$ {\rm MLdegree}(\mathcal{L}^{\nabla f})  \,\, := \,\,
{\rm degree}\bigl(
\mathcal{L}^{\nabla f} \dashrightarrow \C \PP^{c-1}\bigr). $$
Thus the {\em ML degree} of the exponential variety $\mathcal{L}^{\nabla f}$
is the cardinality of the generic fiber of $\pi_\mathcal{L}$ restricted to $\mathcal{L}^{\nabla f}$. This is the algebraic degree of the function that takes points in
$\pi_\mathcal{L}(K)$ back to their unique
 preimage in the nonnegative variety $\, \mathcal{L}^{\nabla f}_{\succeq 0}$.
\end{definition}

Our main result in this section
relates the various notions of degrees:

\begin{theorem}
\label{eq:twoinequalities1}
The following inequalities hold for all exponential varieties:
\begin{equation}
\label{eq:twoinequalities2}
  {\rm MLdegree}(\mathcal{L}^{\nabla f}) \,\, \leq \,\,
{\rm degree}(\mathcal{L}^{\nabla f}) \,\, \leq \,\,
\hbox{the coefficient $\alpha_{c}$ of $\,s^{c-1} t^{d-c}$ in $[X_f]$}.
\end{equation}
Moreover, the left inequality is an equality if and only if $\mathcal{L}^{\nabla f} \cap \mathcal{L}^\perp = \emptyset$.
\end{theorem}

\begin{proof}
Recall from (\ref{eq_points}) that $\alpha_c$ is obtained by counting the points of the intersection of the graph of $\nabla f$ with the product of a generic $(c-1)$-dimensional plane
$L_{c-1}$ and a $(d-c)$-dimensional plane $M_{d-c}$.
If we replace $L_{c-1}$ with a special subspace while
keeping $M_{d-c}$ generic,
  the number of points in the intersection is finite because
  $\dim \mathcal{L}^{\nabla f}=\dim\mathcal{L}$.
  It can only go down
    compared to the completely generic case. This proves the right inequality.

The left inequality is derived from the following well-known fact in algebraic geometry.
When the center of the projection is disjoint from the variety, then the degree of the variety equals the product of the degree of the map and the degree of the image. 
Moreover, when the center of the projection intersects the variety, the product of the cardinality of the generic fiber times the degree of the image is strictly smaller than the degree of the variety. In our situation, the center of the projection
$\pi_{\mathcal{L}}$ is $\mathcal{L}^\perp$, which completes the
proof of the theorem.
\end{proof}

We now offer an illustration of this result with two combinatorial examples.

\begin{example}
\label{eq:dreisechs}
\rm
Let $d = 6$ and $f(\theta)$ be the Laplacian determinant (\ref{eq:K4laplace})
of the complete graph $K_4$.
Using  {\tt Macaulay2} as above, we find that  the gradient multidegree equals
$$ [X_f] \,\, = \,\, {\bf 1} s^5+2 s^4 t+4 s^3 t^2+4 s^2 t^3+2 s t^4+1 t^5 . $$
For instance, we conclude that
all subspaces $\mathcal{L}$ of codimension $2$ or $3$ satisfy
$\,{\rm degree}(\mathcal{L}^{\nabla f}) \leq 4$
In Example~\ref{ex:graphs}, this bound is attained for
the $4$-cycle but not for the $4$-chain.
By contrast, let $g$ be the polynomial obtained from
$f$ by adding $\,\theta_{01} \theta_{02} \theta_{12}
 + \theta_{01} \theta_{03} \theta_{13}
  + \theta_{02} \theta_{03} \theta_{23}
   + \theta_{12} \theta_{13} \theta_{23}$.
 Equivalently, $g$ is the elementary symmetric polynomial
 in six unknowns. We have
 \begin{equation}
 \label{eq:E36}
   [X_g] \,\, = \,\,
{\bf 26} s^5+16 s^4 t+8 s^3 t^2+4 s^2 t^3+2 s t^4+ 1t^5.
\end{equation}
This implies, for instance, that
$\,{\rm MLdegree}(\mathcal{L}^{\nabla g}) \leq
{\rm degree}(\mathcal{L}^{\nabla g}) \leq 16\,$
for all hyperplanes $\mathcal{L}$ in $\C \PP^5$.
We especially note that the gradient map $\nabla g$ is $26$-to-$1$
whereas $\nabla f$ is birational.
\hfill $\diamondsuit$
\end{example}

\begin{example} \rm
All cases of equality in Theorem \ref{eq:twoinequalities1}
occur for  {\em Gaussian graphical models} on four nodes.
Here $d = 10$ and $f(\theta)$ is the determinant
of a symmetric $4 \times 4$-matrix $(\theta_{ij}) $.
\begin{itemize}
\item Let $c = 10$ and $\mathcal{L} = \R^{10}$ the complete model $K_4$.
The degrees in  (\ref{eq:twoinequalities2}) are $\,1 =1=1$.
\item Let $c = 8$ and $\mathcal{L} = \{\theta_{13} {=} \theta_{24} {=} 0\}$ the {\em $4$-cycle}
 model. The degrees in
(\ref{eq:twoinequalities2}) are $5 < 9 = 9 $.
\item Let $c = 3$ and  fix the first model in
\cite[Table 2]{stuhl}. The  degrees in
(\ref{eq:twoinequalities2}) are $\,5 = 5 < 9 $.

\item Let $c = 7$ and $\mathcal{L} = \{ \theta_{13} {=} \theta_{24} {=} \theta_{14} {=}  0\}$
the {\em $4$-chain}. The degrees in
(\ref{eq:twoinequalities2}) are $\,1 < 5 < 17 $.
\end{itemize}
\end{example}

In the remainder of this section, we  examine
the genericity conditions under which equality holds
in the two inequalities in (\ref{eq:twoinequalities2}).
It is immediate from the definition of
the gradient multidegree that ${\rm degree}(\mathcal{L}^{\nabla f})=\alpha_{c}$ for generic
subspaces $\mathcal{L}$. From the proof of Theorem~\ref{eq:twoinequalities1} we know that the left inequality is an equality if and only if $\mathcal{L}^{\nabla f} \cap \mathcal{L}^\perp = \emptyset$.
We believe that this holds for
all polynomials $f$ when $\mathcal{L}$ is generic, but we are presently unable to prove it.

\begin{conjecture}
Let $f$ be a hyperbolic polynomial
in $\theta_1,\ldots,\theta_d$.
Then the exponential variety  $\mathcal{L}^{\nabla f}$
defined by a generic linear subspace $\mathcal{L}$
in $\C \PP^{d-1}$ has its  degree equal to its ML degree.
 Equivalently, for generic $\mathcal{L}$ we have
 $\,\mathcal{L}^{\nabla f} \cap \mathcal{L}^\perp = \emptyset$.
 \end{conjecture}

In the following, we provide a partial result concerning the
set-theoretic image, here denoted $(\nabla f)(\mathcal{L})$.
This image is a dense subset of $\mathcal{L}^{\nabla f}$,
but in general they may differ.

\begin{proposition}
The image $(\nabla f)(\mathcal{L})$
 is disjoint from $\mathcal{L}^\perp$ if and only if $\mathcal{L}$ is not contained in
 the hyperplane tangent
to the hypersurface
 $\{f=0\}$ at a smooth point belonging to $\mathcal{L}$.
\end{proposition}

\begin{proof}
For any $\tau\in \C \PP^{d-1}$ that does not belong to the singular locus of
$\{f=0\}$, we have
$$\sum_i \tau_i \frac{\partial f}{\partial \theta_i}(\tau)=
0\quad\Longleftrightarrow\quad f(\tau)=0.$$
Writing $F = \nabla f$, we have
$$ F(\tau)\in \tau^\perp\quad\Longleftrightarrow\quad f(\tau)=0.$$
To prove the `only if' direction,
consider a smooth point $\tau$ in $\{f=0\}$. The tangent hyperplane
at $\tau$ is $F(\tau)^\perp$.
If $\tau$ lies in $\mathcal{L}$ and
 $\mathcal{L} $ is contained in $ F(\tau)^\perp$,  then
 $F(\tau)\in \mathcal{L}^\perp\cap F(\mathcal{L})$. For the reverse direction let $\sigma\in
 F(\mathcal{L})\cap \mathcal{L}^\perp$. Then $\sigma=F(\tau)$ for some $\tau\in \mathcal{L}$. Hence, $F(\tau)\in \tau^\perp$ and therefore $f(\tau)=0$. This means that $\mathcal{L}$ is contained in the
 tangent hyperplane to $\{f = 0\}$ at $\tau$.
\end{proof}

This proposition implies that the image $(\nabla f)(\mathcal{L})$ is disjoint from $\mathcal{L}^\perp$ if and only if the singular locus of the intersection $\mathcal{L} \cap \{f = 0\}$
is contained in the singular locus of $\{f=0\}$. This is true for generic $\mathcal{L}$ by Bertini's Theorem.
The following corollary identifies cases when
the gradient map $\nabla f $ is regular on $\mathcal{L}$.
In such cases, the set-theoretic image $\nabla(f)(\mathcal{L})$ will be closed
and hence equal to $\mathcal{L}^{\nabla f}$,
so we can conclude that it is generically disjoint from $\mathcal{L}^\perp$.

\begin{corollary}
If the subspace $\mathcal{L}$ is generic of dimension strictly smaller than the codimension
of the singular locus of  $\{f=0\}$, then the ML degree equals the degree of $\mathcal{L}^{\nabla f}$.
\end{corollary}

\section{Elementary Symmetric Polynomials}
\label{sec:ele_sym_pol}

Elementary symmetric polynomials form an important class of 
hyperbolic polynomials \cite{BrandenES}.
In this section, we study the exponential varieties
obtained by slicing their
hyperbolicity cones by a generic $\mathcal{L}$.
Our main result is the formula for the gradient
multidegree in Theorem~\ref{thm:mainelementary}.

We write $E_m$ for the elementary symmetric polynomial
of degree $m$ in the $d$ unknowns $\theta_1,\dots,\theta_d$ , and we consider its
gradient map $\,\nabla E_m:\C\PP^{d-1}\dashrightarrow \C\PP^{d-1}$.
The $i$-th coordinate of $\nabla E_m$ is denoted  $E_{m}^i=\partial_{\theta_i} E_m$.
This is  the elementary symmetric polynomial
of degree $m-1$ in the  $d-1$ unknowns
   $\theta_1,\dots,\theta_{i-1},  \theta_{i+1},\dots,\theta_d$.
  The hyperbolicity cone of $E_m$ equals
$$ C \,\, = \,\, \bigl\{\,\theta \in \R^d: \,E_j(\theta) > 0 \,\,
\hbox{for} \,\,j = 1,2,\ldots,m \,\bigr\}. $$
This defines a hyperbolic exponential family as in Section \ref{sec:dis_gau_hyp}.
The geometry of this family (for $d=4,m=3$) was studied
in Example~\ref{ex:E_3}.
We shall return to this at the end of this section.
See Example \ref{ex:RieszE2} and its pointer to \cite{sokal}
for some information on its Riesz kernel.

The exponential family discussed here
is not an instance of Example \ref{ex:lineargaussian}.
Indeed, by  \cite[Example 5.10]{KPV},
the elementary symmetric polynomial $E_m$
is not a symmetric determinant
whose entries are linear forms, provided $2 \leq m \leq d-2$.
Hence our model in this section is not a linear Gaussian concentration
model. In particular, the theory in  \cite{stuhl} is not applicable.
However, a multiple of $E_m$ is a 
 symmetric determinant with linear entries, as shown in \cite{BrandenES}.

\begin{definition} \rm
The {\em Eulerian number} $A(r,s)$ is the number of permutations of
$\{1,2,\ldots,r\}$ with precisely $s$ ascents.
We have $A(r,0) = A(r,r-1) = 1$ and the following recursion holds:
$$A(r,s) \,\,\,= \,\,\,(r-s) \cdot A(r-1,s-1) \,+\, (s+1)\cdot A(r-1,s) \qquad \hbox{for} \,\,\,1 \leq s \leq r-2.  $$
The Eulerian numbers for small $r$ are
$\,A(3,\bullet) = 1,4,1$,
$ \,A(4, \bullet) = 1,11,11,1$,
$\,A (5,\bullet) = 1,26,66,26,1$,
$\,\,A (6,\bullet) = 1,57,302,302,57,1$, \ and
$\,A(7,\bullet)  = 1,120,1191,2416,1191,120,1$.
\end{definition}

The gradient multidegree of $E_m$ is given by
a  formula in terms of Eulerian numbers:

\begin{theorem}\label{thm:mainelementary}
The gradient multidegree of the elementary symmetric polynomial equals
$$ \qquad \qquad [X_{E_m}] \,\,\, = \,\,\, \sum_{i=1}^d \alpha_i \, s^{i-1}\,t^{d-i},
\qquad \qquad \hbox{where} $$ 
$$
\alpha_i=
\begin{cases}
(m-1)^{i-1}&\text{ for }i<d-m+3\\
\sum_{j=0}^{d-m} (d-m+1-j){{d-1-j}\choose {d-i}}(m-1)^j A(i-2-j,m-d+i-3)&\text{ for }i\geq d-m+3.
\end{cases}
$$
Moreover, the gradient degree of $E_m$ is just the Eulerian number $\,\alpha_d = A(d-1,m-2)$.
\end{theorem}

For example, the formula for $d=6,m=3$ is in (\ref{eq:E36}).
For $d=7$ we find
\begin{equation}
\label{eq:d7ex}
 \begin{matrix}
[X_{E_2}] & = & 1 s^6 \,+\,  1 s^5 t \,+\, 1  s^4 t^2  \,+\, 1 s^3 t^3
 \,+\,  1 s^2 t^4 \,+\, 1  s t^5 \,+\, 1 t^6 ,\\
[X_{E_3}] & = & {\bf 57} s^6 \,+\,  32 s^5 t \,+\, 16  s^4 t^2  \,+\, 8 s^3 t^3
 \,+\,  4 s^2 t^4 \,+\, 2  s t^5 \,+\, 1 t^6 ,\\
[X_{E_4}] & = & 302s^6 \,+\,  {\bf 222} s^5 t \,+\, 81  s^4 t^2  \,+\, 27 s^3 t^3
\,+\,  9 s^2 t^4 \,+\, 3 s t^5 \,+\, 1 t^6 ,\\
[X_{E_5}] & = & 302s^6\,+\,422s^5t\,+\,{\bf 221}s^4t^2\,+
\, 64s^3t^3 \,+ \,16s^2t^4 \,+ \, 4st^5 \, + \,t^6, \\
[X_{E_6}] & = & 57 s^6 \,+\,  157 s^5 t \,+\, 170  s^4 t^2  \,+\, {\bf 90} s^3 t^3
 \,+\,  25 s^2 t^4 \,+\, 5  s t^5 \,+\, 1 t^6 ,\\
[X_{E_7}] & = & 1 s^6 \,+\,  6 s^5 t \,+\, 15  s^4 t^2  \,+\, 20 s^3 t^3
 \,+\,  {\bf 15} s^2 t^4 \,+\, 6  s t^5 \,+\, 1 t^6 .\\
\end{matrix}
\end{equation}

The numbers $\alpha_i$ in Theorem~\ref{thm:mainelementary}
are instances of the {\em mixed Eulerian numbers}. We use the following
geometric interpretation of these numbers.
The {\em $k$-th hypersimplex} is the polytope
$$ \Delta_k \,= \, {\rm conv}
\bigl\{ {\bf e}_{i_1} + {\bf e}_{i_2} + \cdots + {\bf e}_{i_d}\,:\,
1 \leq i_{1}<\cdots<i_{k}\leq d \bigr\} \quad \subset \,\,\, \R^d.
$$
This is a $(d-1)$-dimensional polytope with $\binom{d}{k}$
vertices. For instance, $\Delta_1$ is a $(d-1)$-simplex.
For $s,t > 0$, the Minkowski sum
$ s \Delta_k + t \Delta_1 $ is a polytope of dimension $d-1$.
We consider the volume of $ s \Delta_k + t \Delta_1 $
with respect to the normalized Lebesgue measure that is defined by
 ${\rm vol}(\Delta_1) = 1$.
 This volume is a  homogeneous polynomial  in $s$ and $t$ of degree $d-1$.

\begin{lemma} \label{lem:mixedeuler}
The combinatorial numbers $\alpha_i$ in
Theorem \ref{thm:mainelementary} satisfy
$$ {\rm vol}(s \Delta_{m-1} + t \Delta_1) \,\,\, = \,\,\, \sum_{i=1}^d \alpha_i
\binom{d-1}{i-1}
s^{i-1} t^{d-i} . $$
Equivalently,
$\alpha_i$ equals the {\em mixed volume} of the polytopes
$(\Delta_{m-1},\dots,\Delta_{m-1},\Delta_1,\dots,\Delta_1)$,
 where  the
 hypersimplex $\Delta_{m-1}$ appears $i-1$ times and
 the simplex $\Delta_1$ appears $d-i$ times.
\end{lemma}

\begin{proof}
For $i=d$ this implies ${\rm vol}(\Delta_k) = A(d-1,k-1)$.
In words, the normalized volume of the hypersimplex
is the Eulerian number. This is a classical result due to Laplace.
   Mixed volumes of hypersimplices were studied recently in
\cite{Croitoru,Postnikov}.
The formula we need is that the mixed volume of
$i-1$ copies of $\Delta_{m-1}$ and $d-i$ copies of $\Delta_1$ for $i\geq d-m+3$~equals
 \begin{equation}
 \label{eq:altsum} \sum_{j=0}^{d-m} (d-m+1-j){{d-1-j}\choose {d-i}}(m-1)^j A(i-2-j,m-d+i-3),
 \end{equation}
where $\Delta_{m-1}$ appears $i-1$ times and $\Delta_1$ appears $d-i$ times.
Croitoru gave a recursive formula for mixed Eulerian numbers in
 \cite[Theorem 2.4.6]{Croitoru}. One checks that the sum in
  (\ref{eq:altsum}) satisfies both,
 Croitoru's recursion and initial conditions. This proves the lemma.
 \end{proof}

Thus, the content of  Theorem \ref{thm:mainelementary} is
that the coefficients of the gradient multidegree $[X_{E_m}]$
are the mixed volumes in Lemma \ref{lem:mixedeuler}.
We shall prove this using toric geometry \cite{FultonTV}.


We decompose the gradient map as $\nabla E_m=\pi\circ f$,
where  $f:\C\PP^{d-1}\dashrightarrow \C\PP^{{d\choose {m-1}}-1}$ is the map
given by all square-free monomials of degree $m-1$, and
$\pi:\C\PP^{{d\choose {m-1}}-1}\dashrightarrow \C\PP^{d-1}$ is the
linear projection given by summing the monomials in $\nabla E_m = (E^1_m,E^2_m,\ldots,E^d_m)$.
The closure of the image of $f$ is the toric variety $X_{\Delta_{m-1}}$ associated to the
hypersimplex $\Delta_{m-1}$.

\begin{lemma}\label{lem:piwelldefined}
The linear projection $\pi$ restricts to a regular map $\,X_{\Delta_{m-1}} \rightarrow \C \PP^{d-1}$.
\end{lemma}

\begin{proof}
We first argue that $\pi$ has no base points on the dense torus
of $ X_{\Delta_{m-1}}$. Equivalently,
the homogeneous equations
$\,E_m^1 (\theta) = E_m^2(\theta) = \cdots = E_m^d(\theta) = 0$
have no solutions $\theta$ with all coordinates nonzero.
We use induction on $m$, the case $m=1$ being trivial.
Assume the statement is true for $m$ and consider the equations $E_{m+1}^i(\theta)=0$
for $i=1,\ldots,d$. Summing all the equations we obtain $E_{m}(\theta)=0$. Notice that $E_{m}(\theta)-E_{m+1}^i(\theta)=\theta_iE_m^i(\theta)$. Hence, if $\theta_i\neq 0$ we obtain $E_m^i(\theta)=0$, which allows us to conclude the proof by induction.

We next consider the other orbits on the toric variety $X_{\Delta_{m-1}}$.
Each of these corresponds to a proper face $H$ of $\Delta_{m-1}$.
Each face $H$ is obtained by fixing some of the coordinates at $0$ or $1$, so it is a
hypersimplex of smaller dimension.
Suppose the torus orbit corresponding to $H$ meets the center of projection
of $\pi$ in a point. Then, by considering its nonzero coordinates,
we obtain an inconsistent system of equations, just like in the
first paragraph above.
\end{proof}

\begin{corollary}\label{cor:singloc}
The singular locus of the hypersurface $\{E_m=0\}$ is
the union of all  $d\choose{m-2}$ coordinate planes
of dimension $m-3$ in $\C \PP^{d-1}$. Its
ideal $ \langle E_m^1, E_m^2,\ldots, E_m^d \rangle$
defines a scheme that  is nonreduced for $m \geq 4$, but
smooth outside of the intersection of any two components.
\end{corollary}

\begin{proof}
The fact that the scheme is smooth outside of the intersection of any two components follows from the computation of the tangent space. The scheme is nonreduced, because the
ideal generators span  a $d$-dimensional vector space, while the vector space of
forms of degree $m-1$ in the radical ideal of
the coordinate subspace arrangement has
dimension $d\choose{m-1}$.
\end{proof}

\begin{example} \rm
Let $d=5$ and $m=4$. The gradient ideal
$\langle \nabla E_4 \rangle = \langle E_4^1, E_4^2, E_4^3, E_4^4,E_4^5 \rangle$
is the intersection of its ten minimal primes
$\langle \theta_i,\theta_j,\theta_k \rangle$
and the five embedded primary ideals
$$ \bigl\langle
\,\theta_i^2\,, \,\theta_j^2 \,,\, \theta_k^2 \,,\, \theta_l^2\,,\,
\theta_i \theta_j + \theta_i \theta_k + \theta_j \theta_k,\,
\theta_i \theta_j + \theta_i \theta_l + \theta_j \theta_l,\,
\theta_i \theta_k + \theta_i \theta_l + \theta_k \theta_l,\,
\theta_j \theta_k + \theta_j \theta_l + \theta_k \theta_l \,\bigr\rangle. $$
This defines the ten coordinate lines in $\C \PP^4$,
with embedded points at their intersections.
\hfill $\diamondsuit$
\end{example}

\begin{proof}[of Theorem  \ref{thm:mainelementary}]
Fix the  polytope $P=\Delta_{m-1}+\Delta_1$. This is the
generalized permutohedron \cite{Postnikov} which is the convex hull of
all points   $(2,1^{m-2},0^{d-m+1})\in \R^d$.
Both polytopes $P$ and $\Delta_{m-1}$ are normal,
since they are the base polytopes of polymatroids \cite[Theorem 6.1]{HH}.
Consider the three projective toric varieties $X_P$, $X_{\Delta_{m-1}}$ and $X_{\Delta_1}\simeq\C\PP^{d-1}$.
We have natural  rational maps between these varieties, induced from the
identification of their dense tori.

Let $L_1$ be the very ample line bundle on $X_{\Delta_{m-1}}$
that is given by the hypersimplex $\Delta_{m-1}$.
Let $S$ be the linear system on $\C \PP^{d-1}$
that is spanned by the $d$ partial derivatives $E_m^i$.
This  induces a linear system $S_1$ on $X_{\Delta_{m-1}}$.
The linear system $S_1$ is contained in $H^0(X_{\Delta_{m-1}},L_1)$.
From  Lemma~\ref{lem:piwelldefined} we
know that   $S_1$ is base point free.

The normal fan of $P$ is a common refinement of the normal fans of $\Delta_{m-1}$ and $\Delta_1$. Hence, there are regular toric maps from $X_P$ to $X_{\Delta_{m-1}}$ and $X_{\Delta_1}$. Pulling back the line bundle $L_1$ (resp. $\mathcal{O}_{\PP^{d-1}}(1)$) from $X_{\Delta_{m-1}}$ (resp.~$X_{\Delta_1}$),
we obtain a line bundle $\tilde L_1$ (resp. $\tilde L_2$) on $X_P$. The pull-back of the system $S_1$ (resp. $H^0(\mathcal{O}(1)$)) is a base point free system $\tilde S_1$ (resp. $\tilde S_2$).

As both, $\tilde S_1$ and $\tilde S_2$, are base point free, the intersection of the zero locus of $i-1$ generic sections of $\tilde S_1$ and $d-i$ generic sections of $\tilde S_2$ must be contained in the dense torus of $X_P$. The tori in all three toric varieties are naturally isomorphic. Hence the common zero locus of these sections corresponds to points of the torus of $\C\PP^{d-1}$ where $i-1$ generic sections of $S$ intersect with $d-i$ generic hyperplanes. By Bertini's Theorem, the intersection in the torus consists of smooth points only. Their number is the coefficient of $s^{i-1} t^{d-i}$ in $[X_{E_{m}}]$.

From Bernstein's Theorem \cite[Section 5.4]{FultonTV} we know
that the intersection number of these divisors on $X_P$
is equal to the mixed volume of
 $d-i$ copies of the simplex $\Delta_1$ and
$i-1$ copies of the hypersimplex $\Delta_{m-1}$.
With this, Lemma \ref{lem:mixedeuler} concludes the proof of
Theorem~\ref{thm:mainelementary}.
\end{proof}

It is instructive to examine the various cases
in our formula  in Theorem \ref{thm:mainelementary}.
For $i=d$, we are just looking at $d-1$ copies of the hypersimplex $\Delta_{m-1}$,
and the mixed volume becomes $\alpha_d = {\rm vol}(\Delta_{m-1}) = A(d-1,m-2)$.
At the extreme, for $i=1$, we obtain $\alpha_1 = {\rm vol}(\Delta_1) = 1$.

The case distinction in our formula for $\alpha_i$
is understood via the singular locus of $\{E_m = 0\}$.
According to Corollary \ref{cor:singloc}, this has dimension
$m-3$. If $d-i > m-3$ then the $d-i$ generic hyperplanes in $\C \PP^{d-1}$
will miss this singular locus. In this case, B\'ezout's Theorem
applies to the $i-1$ generic sections of the linear system $S$
in $\C \PP^{d-1}$. They will meet in $(m-1)^{i-1}$ distinct reduced points.
This explains the formula $\alpha_i=(m-1)^{i-1}$ for $i < d-m+3$.

Of special interest is the borderline case $i= d-m+3$.
Here the $m-4$ generic hyperplanes meet the singular locus
in $\binom{d}{m-2}$ distinct reduced points, one for each
coordinate subspace, by Corollary \ref{cor:singloc}.
This number gets subtracted from the B\'ezout number, and therefore
\begin{equation}
\label{eq:borderline}
 \alpha_{d-m+3}\,\,=\,\,(m-1)^{d-m+2}\,-\,{d\choose {m-2}}.
\end{equation}
This explains the bold face numbers along the borderline diagonal
for $m=7$ in (\ref{eq:d7ex}).

For $i>d-m+3$ the intersection
is not proper in $\C \PP^{d-1}$.
It contains the base locus of $S$,
which is the singular locus of $\{E_m = 0\}$,
whose dimension strictly dominates the dimension of the other components. This is an
instance of \emph{excess intersection}.
The same issue occurs for other hyperbolic
polynomials $f(\theta)$.
For the Gaussian case, when $f(\theta)$ is the
symmetric determinant, this is addressed
 in \cite[Theorem 2.3]{stuhl}.
For the elementary symmetric polynomial,
we resolved this problem by passing to the
toric variety $X_{\Delta_{m-1}}$. This worked well
here, thanks to Lemma \ref{lem:piwelldefined}.
However,  for other $f(\theta)$, the toric approach will not suffice.

\begin{example} \rm
Let $d=6$, $m=3$ and revisit Example \ref{eq:dreisechs}.
The gradient degree of $g(\theta) = E_3(\theta)$
is $ {\rm vol}(\Delta_2) = A(5,1) = 26$, which is the leading coefficient
in (\ref{eq:E36}). The Laplacian determinant $f(\theta)$
 is obtained from $g(\theta)$ by deleting
four of the $20$ terms. Both $\nabla f$ and $\nabla g$
have the same Newton polytope, namely the second
hypersimplex $\Delta_2$. However, $\nabla f$ is birational: the gradient degree of $f(\theta)$ is $1$.
Lemma \ref{lem:piwelldefined} holds for $g(\theta)$ but it
fails dramatically for  $f(\theta)$.
\hfill $\diamondsuit$
\end{example}

Theorem~\ref{thm:mainelementary}
provides an upper bound for the degree and ML degree
of the exponential variety $\mathcal{L}^{\nabla E_m}$
defined by any linear subspace $\mathcal{L} \subset \R^d$,
by  Theorem \ref{eq:twoinequalities1}.
The more special we choose $\mathcal{L}$, the further
away we expect to be from the upper bound $\alpha_d$.
In the next example we explore the range of possibilities
for the exponential family seen in Example~\ref{ex:E_3}.

\begin{example}
\label{ex:E_3c}
\rm
Let $d=4$ and $m=3$. The cubic surface $\{E_3 = 0\}$ has four singular points.
It bounds the convex set $C$ on the left in Figure \ref{fig:somosa}.
The dual body $K$, on the right in Figure \ref{fig:somosa},
is the convex hull of the surface $\{Q = 0\}$ in (\ref{eq:steiner}).
The gradient multidegree of $E_3$ equals
$$ [X_3]\,\, = \,\, 4 s^3 \,+ \,4 s^2 t \, +\,2 s t^2 \,+\, t^3. $$
The leading coefficient tells us that
the map $C \rightarrow K$ is $4$-to-$1$.
 The second coefficient reveals
that the surface $\mathcal{L}^{\nabla E_3}$ has degree $4$
whenever $\mathcal{L}$ is a random plane in $\R \PP^3$.
However, when the plane $\mathcal{L}$ meets the singular locus of
$\{E_3 = 0\}$ then the   degree of that surface will drop.

\begin{itemize}
\item The plane $\mathcal{L} = \{\theta_1 + \theta_2 = 2 \theta_3\}$ contains
one singular point. The variety $\mathcal{L}^{\nabla E_3}$ is a cubic surface
that is singular along the line $\{
4 \sigma_1 - \sigma_3 + \sigma_4 =
4 \sigma_2 - \sigma_3 + \sigma_4 = 0\}$.

 \item The plane $\mathcal{L} = \{\theta_1+\theta_2+\theta_3=0\}$
 intersects $\{E_3 = 0\}$ in exactly the same singular point.
   But now the resulting exponential surface is a plane:
$\mathcal{L}^{\nabla E_3} = \{ \sigma_1+ \sigma_2+ \sigma_3= \sigma_4\}$.

\item The plane $\mathcal{L} = \{\theta_1 = 2 \theta_2 \}$ meets $\{E_3 = 0\}$
in two singular points. Now $\mathcal{L}^{\nabla E_3}$ is the quadric
$\{
45 \sigma_1^2-63 \sigma_1 \sigma_2+18 \sigma_2^2+18 \sigma_1 \sigma_3
-9 \sigma_2 \sigma_3+\sigma_3^2+18 \sigma_1 \sigma_4-9 \sigma_2 \sigma_4-2 \sigma_3 \sigma_4+\sigma_4^2 = 0\}$.

\item The plane $\mathcal{L} = \{\theta_1 = 0\}$ spanned by three singular points
maps to the facet
$\mathcal{L}^{\nabla E_3} = \{\sigma_1 = \sigma_2+\sigma_3 +\sigma_4\}$
of the octahedron inside $K$ that was mentioned in Example~\ref{ex:E_3}.
\end{itemize}

\begin{figure}[h]
\centering
 \includegraphics[width=7.6cm]{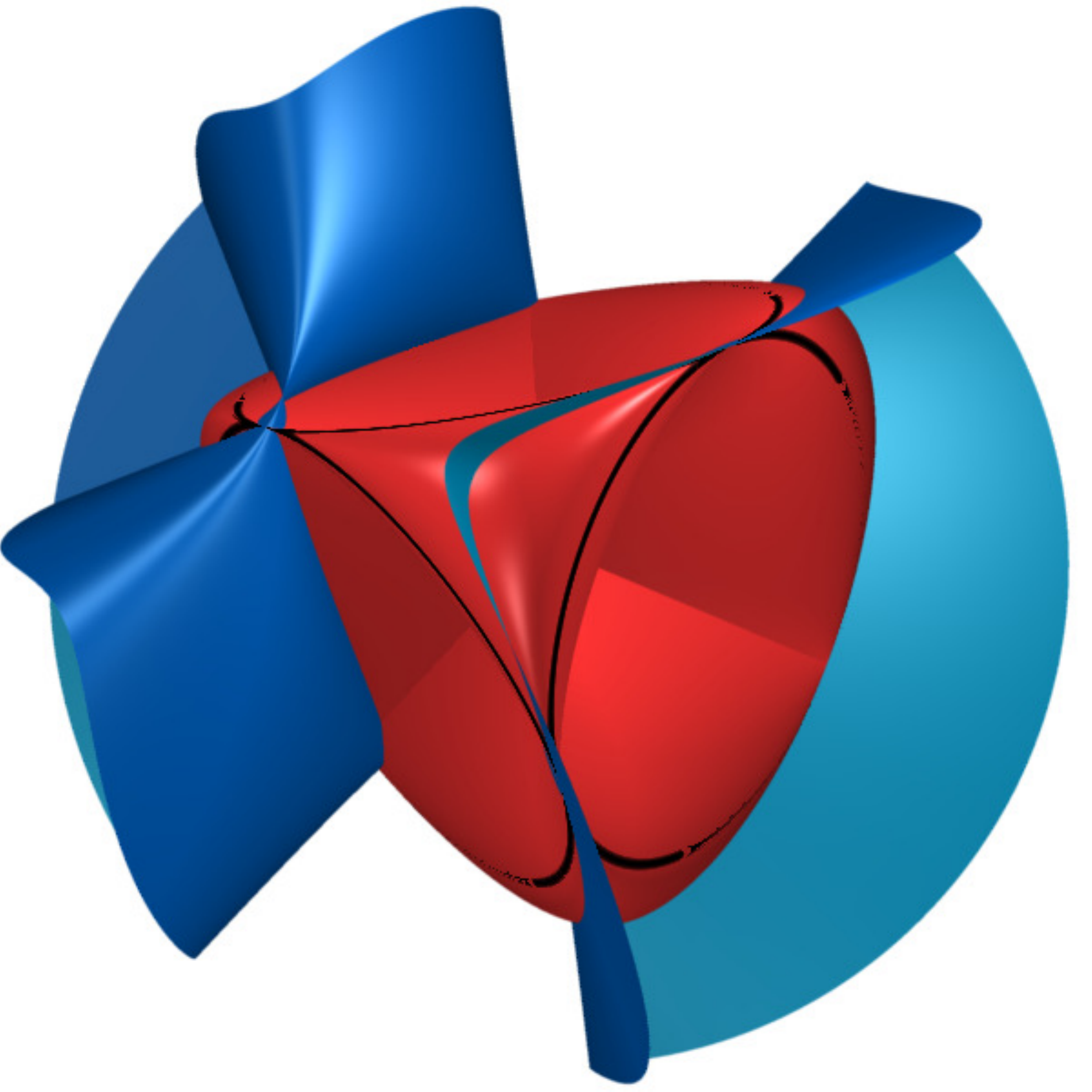}  \qquad\includegraphics[width=6.1cm]{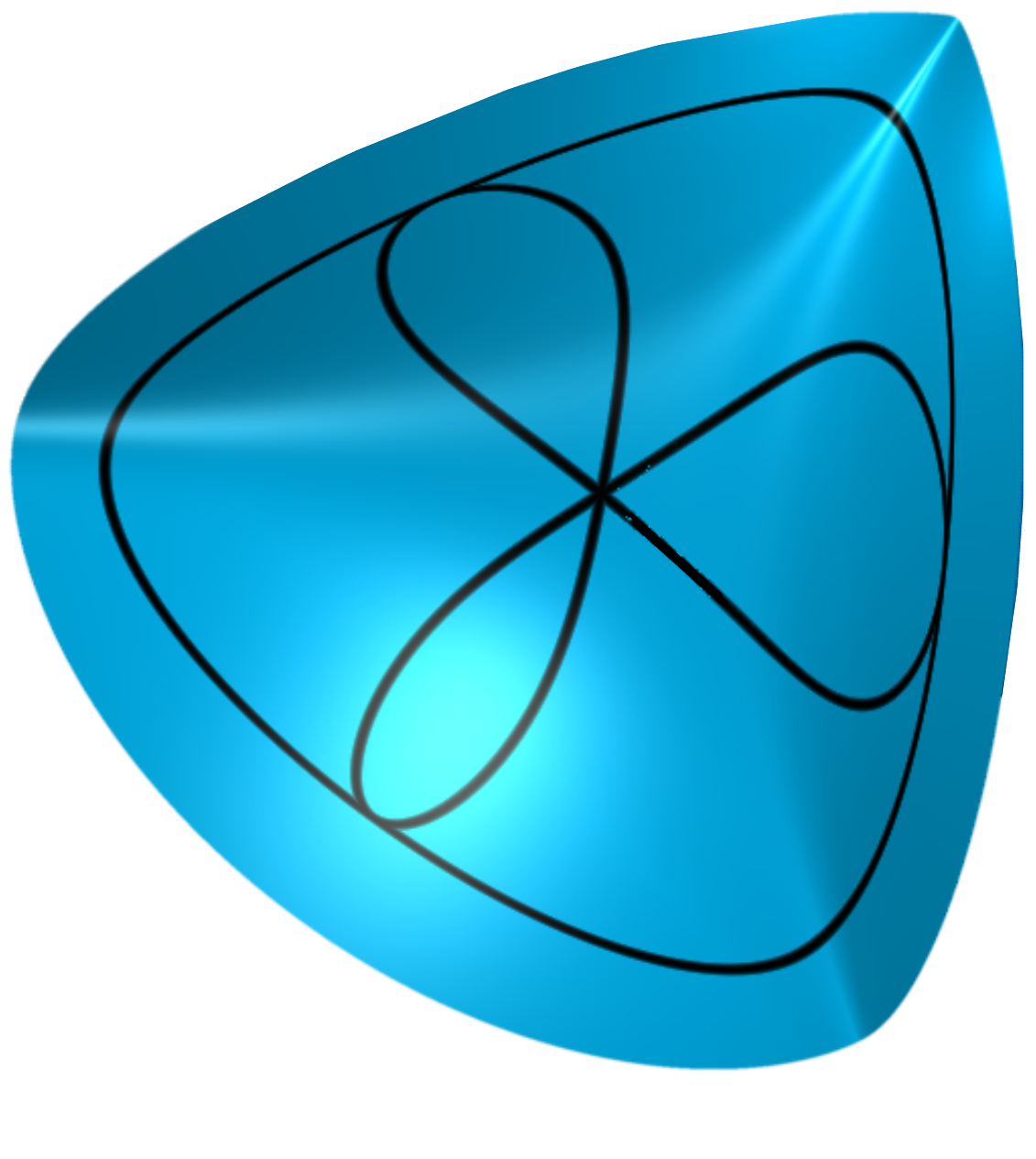}
 \caption{The positive exponential variety is one connected component in the
intersection of the exponential variety (blue) with the convex body of sufficient statistics (red).}
\label{fig:redblue}
\end{figure}

The generic case is exhibited by the
 plane in Figure~\ref{fig:yellowgreen}, which is
$\mathcal{L} = \{\theta_1 + 2 \theta_2 + 3 \theta_3 = 6 \theta_4\}$.
Figure~\ref{fig:redblue} is the dual of Figure~\ref{fig:yellowgreen}.
The exponential variety $\mathcal{L}^{\nabla E_3}$
is the quartic surface shown in blue.
We see that  $\mathcal{L}^{\nabla E_3}$
intersects the convex body $K$ (shown in red in Figure~\ref{fig:redblue}) in two
 connected components. The positive exponential variety $\mathcal{L}^{\nabla E_3}_{\succ 0}$
is the component that cuts off the rightmost red circle
in Figure~\ref{fig:redblue} (left) from the other three.
That component is shown in Figure~\ref{fig:redblue} (right),
together with the positive exponential variety $\mathcal{L}^{\nabla E_3}_{\succ 0}$.
The curve indicates where $\mathcal{L}^{\nabla E_3}$ intersects the Steiner surface.
It has degree six.
The outer curve bounds $\mathcal{L}^{\nabla E_3}_{\succ 0}$. The projection $\pi_\mathcal{L}$ flattens out $\mathcal{L}^{\nabla E_3}_{\succ 0}$ and maps it bijectively onto  the
planar convex set $\mathcal{K}_\mathcal{L}$.
\smallskip

Next consider the case when $\mathcal{L}$ is a line in $\C \PP^3$.
Here $[X_3]$ tells us that the curve $\mathcal{L}^{\nabla E_3}$ is expected to have degree $2$.
This drops to $1$ if $\mathcal{L}$ contains one of the four singular points.
However, the image  $\mathcal{L}^{\nabla E_3}$ can also degenerate to a double line
for special lines $\mathcal{L}$ that are disjoint from the singular locus of $\{E_3 = 0\}$.
For such a line $\mathcal{L}$, the restriction of the map $\nabla E_3$ to $\mathcal{L}$
has degree $2$.
One such special line is
$ \mathcal{L} = \{\theta_1 = \theta_2 , \theta_3 = \theta_4\}$,
which maps to
$\mathcal{L}^{\nabla E_3}
= \{\sigma_1 = \sigma_2 , \sigma_3 = \sigma_4\}$.
Finally, if $\mathcal{L}$ is one of the six coordinate lines, then
$\mathcal{L}^{\nabla E_3}$ is a point. In our picture this means that the
edges of the tetrahedron in $C$ contract to the vertices
of the octahedron in $K$.
\hfill $\diamondsuit$
\end{example}

\begin{example} \rm
Let $d = 5$. It is instructive to ponder the geometry
of the convex bodies $C$ and $ K$
 and how these $4$-dimensional bodies
 encompass their various exponential varieties.

First consider  $m=4$. Here the situation is analogous
to Examples~\ref{ex:E_3} and \ref{ex:E_3c}.
The body $K$ consists of all points $\sigma \in \R \PP^4$ where
$\sigma_5$ is the sum of all entries in a
positive semidefinite $4 {\times} 4$-matrix with diagonal
$(\sigma_1,\sigma_2,\sigma_3,\sigma_4)$.
The map from $C$ to $K$ is $11$-to-$1$, since
$$ [X_4] \,\,=\,\, 11s^4 \,+\,17s^3t\,+\,9s^2t^2 \,+ \,3st^3 \, + \,t^4 .$$
The main component in the algebraic boundary of $K$ is
the hypersurface dual to $\{E_4 = 0\}$.
This threefold has degree $8$ and it touches
each of the five facets of the ambient $4$-simplex
in a Steiner surface (\ref{eq:steiner}).
For a general hyperplane $\mathcal{L} \subset \C \PP^4$,
the exponential variety $\mathcal{L}^{\nabla E_4}$ is a threefold
of degree $17$. If $\mathcal{L}$ is a general  $2$-plane, then
$\mathcal{L}^{\nabla E_4}$ is a surface of degree $9$.

Next consider $m=3$. Now $K$ is strictly contained in the previous body, and
$\partial K$ meets each tetrahedral facet of the ambient $4$-simplex in an inscribed sphere,
such as
$\{\sigma_1^2+\sigma_2^2+\sigma_3^2+\sigma_4^2 =
\sigma_1 \sigma_2+\sigma_1\sigma_3+\sigma_1\sigma_4+\sigma_2\sigma_3
+\sigma_2\sigma_4
+\sigma_3\sigma_4\}$. The map from $C$ to $K$ is $11$-to-$1$, since
$$ [X_3] \,\,=\,\, 11s^4\,+\,8s^3t\,+\,4s^2t^2\,+\,2st^3\,+\,t^4 . $$
The main component in $\partial K$ is
the hypersurface dual to $\{E_3 = 0\}$, which is
defined by a polynomial of degree $14$ with $3060$ terms.
Setting $\sigma_5 = 0$, the resulting surface decomposes as twice the inscribed sphere
times an irreducible surface of degree $10$. This suggests that there
is lots of interesting geometry yet to be discovered in the cones of
sufficient statistics for the hyperbolic exponential families
given by elementary symmetric polynomials for $d \geq 5$.~\hfill~$\diamondsuit$
\end{example}

\section{Hankel Models}
\label{sec:hankel}

The paradigm of an exponential variety is the
 set of symmetric matrices whose inverses satisfy linear constraints.
These arise from the exponential family structure on the
multivariate normal distributions. Here the hyperbolic polynomial
$f(\theta)$ is the determinant of a symmetric $m \times m$-matrix of unknowns.
The resulting {\em linear Gaussian concentration models} were
studied in \cite{andersonLinearCovariance} and their geometry in \cite{stuhl}. In this paper we encountered them  in
Example~\ref{ex:lineargaussian}.

In this section we present a beautiful instance
that is related to numerous  topics
in the mathematical sciences.
Let $\mathbb{S}^m \simeq \mathbb{R}^{\binom{m+1}{2}}$ denote the space of symmetric
$m \times m$-matrices, and $C = {\rm PD}_m$ the positive definite cone.
Let $\mathcal{L}$  be the $(2m-1)$-dimensional subspace of
all {\em Hankel matrices} $(\theta_{i+j})_{1 \leq i,j \leq m}$.
Then  $C_{\mathcal{L}} =  C \cap \mathcal{L}$ is the cone of
positive definite Hankel matrices.

\begin{proposition}
The closure of  $C_\mathcal{L} $ consists of the vectors of
 $2m-1$ first moments of all probability distributions on $\mathbb{R}$.
Its dual  $K_\mathcal{L} = (C_\mathcal{L})^\vee$ is the cone of all nonnegative polynomials of
degree $2m-2$ in one variable~$x$. The polynomials in
$K_\mathcal{L}$ are precisely 
the sums of squares.
\end{proposition}

\begin{proof}
The first statement is \cite[Theorem 3.146]{BPT}.
For the second statement see
 \cite[Section 3.5.1]{BPT}.
 Every nonnegative real polynomial in one variable $x$  is a sum of squares
 in   $\R[x]$.
\end{proof}

We shall see in Proposition \ref{prop:GrassBezout1}
that the exponential variety $\mathcal{L}^{\nabla f}$ equals
 the Grassmannian ${\rm Gr}(2,m+1)$
of $2$-dimensional linear subspaces of $\mathbb{R}^{m+1}$,
in its Pl\"ucker embedding in $\mathbb{R} \PP^{{\binom{m+1}{2}}-1}$.
The  positive Grassmannian
  $\,\mathcal{L}^{\nabla f}_{\succ 0}=    {\rm Gr}(2,m+1)_{\succ 0}\,$
  is in natural bijection with the cone 
       $K_\mathcal{L}$.

Hankel matrices are closely related to {\em Toeplitz matrices}. In fact, under renaming the
coordinates in $\mathbb{S}^m$, they represent  the same subspace $\mathcal{L}$.
See \cite{gray2006toeplitz} for an introduction to Toeplitz matrices from the engineering perspective.
In statistics, Toeplitz matrices are an instance of a {\em colored Gaussian graphical model} as in \cite{HLcolor} and 
\cite[Section~5]{stuhl}. The results  in this section are stated for Hankel matrices, but
the story would be isomorphic for Toeplitz matrices.

To match the notation with Section~\ref{sec:lin_sub},
we now have $d = \binom{m+1}{2}$ and $c = 2m-1$.
The map $\pi_\mathcal{L} : \R^d \rightarrow \R^c$
that is dual to the inclusion $\mathcal{L} \subset \R^d$
can be described as follows.
We identify the domain $\R^d$ with $\mathbb{S}^m$, now regarded as the ambient
   space for covariance matrices $\Sigma = (\sigma_{ij})$.
   We identify the target $\R^c$ with the space $\R[x]_{\leq 2m-2}$ of
univariate polynomials of degree at most $2m-2$.
Each symmetric matrix $\Sigma$ is mapped to such a polynomial as follows:
\begin{equation}
\label{eq:makepoly}  \pi_\mathcal{L} \,: \, \Sigma \,\,\mapsto \,\,
(1, x , x^2 ,\ldots, x^{m-1}) \cdot \Sigma \cdot (1,x,x^2,\ldots, x^{m-1})^T.
\end{equation}
The image $K_\mathcal{L}$ of the positive semidefinite cone $K = \overline{\rm PD}_m$
under the map $\pi_\mathcal{L}$ consists of all polynomials that are
sums of squares in $\R[x]$. These are
precisely the nonnegative polynomials.

The following result is a reformulation of a well-known result in linear algebra,
stating that the inverses of
Hankel matrices are precisely the B\'ezout matrices; see \cite[Section 6]{HF}.

\begin{proposition}
\label{prop:GrassBezout1}
After a linear change of coordinates in $\mathbb{R}^d \simeq \mathbb{S}^m$,
the exponential variety $\mathcal{L}^{-1}$ of all inverse Hankel matrices
coincides with (the affine cone over) the Grassmannian ${\rm Gr}(2,m+1)$. In particular,
the degree of $\mathcal{L}^{-1}$ in
$\C \PP^{d-1}$ is the Catalan number
$\frac{1}{m}\binom{2m-2}{m-1}$.
\end{proposition}

\begin{proof}
We here write ${\rm Gr}(2,m+1)$ for the
cone over the Grassmannian. This
is an affine variety of dimension $2m-1$
in  the ambient space $\mathbb{R}^d = \mathbb{R}^{\binom{m+1}{2}}$.
 The Pl\"ucker coordinates are denoted
$p_{ij}$ for $0 \leq i < j \leq m$, and these satisfy the quadratic Pl\"ucker relations
\begin{equation}
\label{eq:pluckerrel}
\underline{p_{il} p_{jk}} \, -\, p_{ik} p_{jl} \,+ \,
 p_{ij} p_{kl}  \,\, = \,\, 0  \qquad \hbox{for}\,\,\, 0 \leq  i < j < k < l \leq m.
 \end{equation}
The degree of ${\rm Gr}(2,m+1)$ is the {\em Catalan number}
$\frac{1}{m}\binom{2m-2}{m-1}$, e.g.~by
\cite[Proposition 3.7.4]{AIT}.

For each Pl\"ucker vector $p \in {\rm Gr}(2,m+1)$ we write the
corresponding B\'ezout matrix as $\Sigma = (\sigma_{ij})$.
This $m \times m$ matrix is symmetric and its entries are
thought of as covariances:
\begin{equation}
\label{eq:b_to_p}
\sigma_{ij} \,\,\, =  \sum_{k={\rm max}(1,i+j-m)}^i \!\!\! p_{k,i+j+1-k} \qquad
{\rm for} \,\, 1 \leq i \leq j \leq m .
\end{equation}
These linear equations have an upper triangular structure, so they can be inverted
to express the Pl\"ucker coordinates $p_{ij}$ in terms of  the B\'ezout coordinates $\sigma_{ij}$.
This is the promised linear change of coordinates in
$\R^{\binom{m+1}{2}}$. In particular, we can write
the quadrics (\ref{eq:pluckerrel}) in the $\sigma_{ij}$.
\end{proof}

This defines the structure of an exponential variety on the Grassmannian
${\rm Gr}(2,m+1)$. The positive exponential variety
${\rm Gr}(2,m+1)_{\succ 0}$  consists of all solutions to
(\ref{eq:pluckerrel}) such that $\Sigma = (\sigma_{ij})$ is positive definite.
This model is in bijection with the cone $K_{\mathcal{L}}$ of nonnegative polynomials
of degree at most $2m-2$ under the sufficient statistics map
$\pi_\mathcal{L}$ in  (\ref{eq:makepoly}).

\begin{example} \rm
Let $m=4$. The Pl\"ucker coordinates
in terms of the B\'ezout coordinates are
$$
\begin{matrix}
p_{12} = \sigma_{12}, &
p_{13} = \sigma_{13}, &
p_{14} = \sigma_{14}, &
p_{15} = \sigma_{15}, &
p_{23} = \sigma_{23} - \sigma_{14}, \\
p_{24} = \sigma_{24} - \sigma_{15}, &
p_{25} = \sigma_{25}, &
p_{34} = \sigma_{34} - \sigma_{25}, &
p_{35} = \sigma_{35}, &
p_{45} = \sigma_{45}.
\end{matrix}
$$
After this substitution, the ideal of Pl\"ucker relations for $m=4$ equals
$$
\begin{matrix}
\langle
\sigma_{11} \sigma_{24}{-}\sigma_{11} \sigma_{33}{-}\sigma_{12} \sigma_{14}{+}\sigma_{12} \sigma_{23}{+}\sigma_{13}^2{-}\sigma_{13} \sigma_{22},\,
\sigma_{13} \sigma_{44}{-}\sigma_{14} \sigma_{34}{-}\sigma_{22} \sigma_{44}{+}\sigma_{23} \sigma_{34}{+}\sigma_{24}^2{-}\sigma_{24} \sigma_{33},
\qquad \\
\sigma_{11} \sigma_{34}{-}\sigma_{12} \sigma_{24}{-}\sigma_{13} \sigma_{14}{+}\sigma_{14} \sigma_{22},
\sigma_{11} \sigma_{44}{-}\sigma_{13} \sigma_{24} {-}\sigma_{14}^2{+}\sigma_{14} \sigma_{23} ,
\sigma_{12} \sigma_{44} {-}\sigma_{13} \sigma_{34} {-}\sigma_{14} \sigma_{24}{+}\sigma_{14} \sigma_{33}
\rangle.
\end{matrix}
$$
Thus, for the space $\mathcal{L}$ of $4 \times 4$-Hankel matrices, we have
$\mathcal{L}^{-1} = {\rm Gr}(2,5)$. The positive exponential variety
$\mathcal{L}^{-1}_{\succ 0} = {\rm Gr}(2,5)_{\succ 0}$ is our positive Grassmannian.
It consists of
 all positive definite symmetric $4 \times 4$-matrices
$\Sigma = (\sigma_{ij})$ that satisfy these five quadrics.
The bijection
$$ \Sigma \,\,\,\,\mapsto \,\,\, \sum_{i=1}^4 \sum_{j=1}^4 \sigma_{ij} x^{i+j-2} $$
 takes the set ${\rm Gr}(2,5)_{\succ 0}$ to the
$7$-dimensional convex cone of all nonnegative sextics.
\hfill $\diamondsuit$
\end{example}

Maximum likelihood estimation  means inverting the map
from the positive Grassmannian ${\rm Gr}(2,5)_{\succ 0}$
onto $K_\mathcal{L}$.
This inverse takes a nonnegative polynomial of degree $2m-2$
to  the analytic center of its {\em Gram spectrahedron} \cite{Tsu}.
We prove that this algebraic function from
$K_{\mathcal{L}}$ to $\mathcal{L}^{-1}_{\succeq 0}$ has the expected degree,
meaning that the left inequality in (\ref{eq:twoinequalities2}) holds with equality.

\begin{proposition}
\label{prop:GrassBezout2}
The ML degree of the Grassmannian ${\rm Gr}(2,m+1)$, in its guise
as exponential variety $\mathcal{L}^{-1}$ of $m {\times} m$  inverse Hankel matrices,
is the Catalan number $\frac{1}{m}\binom{2m-2}{m-1}$.
\end{proposition}

\begin{proof}
Our claim states that equality holds for the left inequality in
 Theorem \ref{eq:twoinequalities1} when $f$ is
the symmetric determinant and $\mathcal{L}$ is the space of Hankel matrices.
In order to prove this, we must show that $\mathcal{L}^{\nabla f} = {\rm Gr}(2,m+1)$
is disjoint from $\mathcal{L}^\perp$.
Note that $\mathcal{L}^\perp = \{\ell_1 = \ell_2 = \cdots = \ell_{2m-1} = 0\}$,
where $\ell_i$ is the sum of all unknowns on the $i$-th antidiagonal of the matrix $\sigma$.
Thus $\ell_i = 2 \sum_{j=1}^{i/2}  \sigma_{j,i-j+1}$ when $i$ is even, and
$\ell_i =  \sigma_{\frac{i+1}{2},\frac{i+1}{2}} + 2 \sum_{j=1}^{\lfloor{i/2}\rfloor}  \sigma_{j,i-j+1}$
when $i$ is odd.

We now pass to Pl\"ucker coordinates $p_{ij}$ for the Grassmannian ${\rm Gr}(2,m+1)$.
The  Pl\"ucker relations (\ref{eq:pluckerrel}) are homogeneous with respect to the
$\N$-grading given by ${\rm degree}(p_{jk}) = j+k$. By (\ref{eq:b_to_p}), the linear form $\ell_i$ is a positive integer linear
combination of all Pl\"ucker coordinates of degree $i$. Equivalently,
in the graded poset that underlies the straightening law for ${\rm Gr}(2,m+1)$ as in~\cite[Section 14.3]{CCA},
the linear form $\ell_i$ is a positive sum of all poset elements of height~$i$.
This implies that $\{\ell_1,\ldots,\ell_{2m-1}\}$ is a linear system of parameters
modulo the Stanley-Reisner ideal of that poset. That ideal is generated by the underlined
leading monomials in (\ref{eq:pluckerrel}).
From this it follows that $\{\ell_1,\ldots,\ell_{2m-1}\}$ is a linear system of parameters
modulo the ideal of P\"ucker relations. This is equivalent to
$\,{\rm Gr}(2,m+1) \cap \mathcal{L}^\perp = \emptyset\,$ in
$\,\C \PP^{{\binom{m+1}{2}}-1}$, as desired.
\end{proof}

It would be very interesting to extend this analysis to Hankel
matrices of polynomials in more than one variable.
While we do not yet have any results to report on this problem,
let us still close our paper by describing
the corresponding linear Gaussian concentration model.
In our view, this may become an important
object for convex algebraic geometry~\cite{BPT,stuhl}.

We fix two arbitrary positive integers $r$ and $s$,
and we set $m = \binom{r+s-1}{s-1}$ and
$c = \binom{2r+s-1}{s-1}$. These are the numbers of
monomials in $s$ variables of degree $r$ and $2r$ respectively.
Each matrix $\theta$ in  $\S^m$ has its rows and columns labeled
with the monomials of degree $r$.
A matrix $\theta \in \S^m$ is a {\em Hankel matrix}
if its entries are indexed by the product of the
row label and the column label.
The set $\mathcal{L}$ of all Hankel matrices is
a subspace of dimension $c$ in $\S^m$

For instance, for $r=2, s = 3, m = 6, c = 15$,
with row and column labels written additively as
$(200),(020),(002),(110),(101),(011)$, the space of
Hankel matrices equals
\begin{equation}
\label{eq:generalizedhankel}
\mathcal{L} \quad = \quad  \left\{\,
\begin{pmatrix}
\theta_{400} & \theta_{220} & \theta_{202} & \theta_{310} & \theta_{301} & \theta_{211}  \\
 \theta_{220} & \theta_{040} & \theta_{022} & \theta_{130} & \theta_{121} & \theta_{031} \\
 \theta_{202} & \theta_{022} & \theta_{004} & \theta_{112} & \theta_{103} & \theta_{013} \\
 \theta_{310} & \theta_{130} & \theta_{112} & \theta_{220} & \theta_{211} & \theta_{121} \\
 \theta_{301} & \theta_{121} & \theta_{103} & \theta_{211} & \theta_{202} & \theta_{112} \\
 \theta_{211} & \theta_{031} & \theta_{013} & \theta_{121} & \theta_{112} & \theta_{022}
 \end{pmatrix} \,\,:\,\,
 \theta \in \R^{15}
 \right\}.
\end{equation}
Thus $C_\mathcal{L}$ is the spectrahedral cone consisting of all
positive definite Hankel matrices.

The dual space $\R^d/\mathcal{L}^\perp \simeq \R^c$
is identified with the space of homogeneous polynomials
of degree $2r$ in $s$ variables. The dual cone
$K_\mathcal{L} =  (C_\mathcal{L})^\vee$ consists of
those polynomials that are sums of squares
of polynomials of degree $r$,
by \cite[Corollary 4.36]{BPT}.

 Hilbert's classical theorem (see \cite[Theorem 4.3]{BPT}) states that
 the sum-of-squares cone $K_\mathcal{L}$  coincides with the cone of
polynomials that are nonnegative on $\R^s$ if
and only if $r=1$, $s=2$, or ($r=2$ and $s = 3$).
In all other cases, the latter cone strictly contains $K_\mathcal{L}$.
The exponential variety $\mathcal{L}^{-1}$ consists of all
inverse Hankel matrices, and its positive part
$\mathcal{L}^{-1}_{\succ 0}$ maps bijectively onto
the cone $K_\mathcal{L}$. The inverse map takes
each sum-of-squares polynomial  $p$ to a distinguished
sum-of-squares representation $\hat p$. Namely,
$\hat p$ is the analytic center in the Gram spectrahedron of $p$,
which is the fiber over $p$ of the map
$\pi_{\mathcal{L}}$ from $\overline{{\rm PD}}_m$ onto $K_\mathcal{L}$.
The following problem asks for a higher-dimensional generalization
 of Propositions \ref{prop:GrassBezout1} and \ref{prop:GrassBezout2}.

\begin{problem}
Determine  the degree and the ML degree of  the exponential variety $\mathcal{L}^{-1}$
of inverse Hankel matrices for $s \geq 3$ variables, and study its defining polynomial equations.
\end{problem}

For example, for $r=2,s=3$, we start with the linear subspace
$\mathcal{L} \subset \R \PP^{20}$ given  in (\ref{eq:generalizedhankel}).
 Its exponential variety $\mathcal{L}^{-1}$ consists of the
 inverses of all such Hankel matrices. This is a projective variety
 of dimension $14$ in $\C \PP^{20}$,
 whose positive part maps bijectively onto the
 cone $K_\mathcal{L}$ of all nonnegative ternary quartics.
That cone was described in \cite[Theorem 6.2]{orbitopes}.

\bigskip

\noindent
{\bf Acknowledgments.}
All four authors were hosted by the Simons Institute for the Theory of Computing in Berkeley during
the Fall 2014 program {\em Algorithms and Complexity in Algebraic Geometry}.
We thank Petter Br\"{a}nd\'{e}n, June Huh, Mario Kummer,
Gregorio Malajovich and Donald Richards for helpful discussions.

\begin{small}

\bigskip
\end{small}

\footnotesize
\noindent {\bf Author's addresses:}

\smallskip

\noindent Mateusz Micha{\l}ek, {\tt wajcha@berkeley.edu} \\
Mathematical Institute, Polish Academy of Sciences, Warsaw, 00-656, Poland\\
\noindent Department of Mathematics,
 University of California, Berkeley, CA 94720-3840, USA \\

\smallskip

\noindent Bernd Sturmfels, {\tt bernd@berkeley.edu} \\ Department of Mathematics,
University of California, Berkeley, CA 94720-3840, USA \\

\smallskip

\noindent Caroline Uhler, {\tt caroline.uhler@ist.ac.at} \\
  Institute of Science and Technology Austria,
            3400 Klosterneuburg,
            Austria \\
Massachusetts Institute of Technology, EECS,
Cambridge, MA 02139-4307, USA \\

\smallskip

\noindent Piotr Zwiernik, {\tt piotr.zwiernik@gmail.com} \\
Dipartimento di Matematica, Universit\`a di Genova, 16146 Genova, Italy \\
Department of Economics and Business, Universitat Pompeu Fabra, 08005
Barcelona, Spain \\

\end{document}